\documentclass{amsart}

\usepackage{amsfonts}
\usepackage{amssymb}
\usepackage{amsmath}
\usepackage{mathrsfs}

\newtheorem{theorem}{Theorem}[section]
\newtheorem{lemma}[theorem]{Lemma}

\newtheorem{proposition}[theorem]{Proposition}

\theoremstyle{definition}
\newtheorem{definition}[theorem]{Definition}
\newtheorem{example}[theorem]{Example}

\newtheorem{remark}[theorem]{Remark}

\numberwithin{equation}{section}

\newcommand{\R}{\mathbb R}

\newcommand{\B}{\mathbb B}

\newcommand{\X}{\mathbb X}
\newcommand{\Y}{\mathbb Y}

\newcommand{\nullv}{\mathbf{0}}

\newcommand{\grph}{{\rm graph}\, }
\newcommand{\dom}{{\rm dom}\, }

\newcommand{\inte}{{\rm int}\, }

\newcommand{\Lin}{\mathcal{L}}

\newcommand{\Solv}{\mathcal{S}}

\newcommand{\val}{{\rm val}_{\mathcal{P}}}

\newcommand{\cov}{{\rm cov}\, }

\newcommand{\Fder}{{\rm D}}

\newcommand{\stsl}[1]{|\nabla #1|}     
\newcommand{\sostsl}[1]{\overline{|\nabla #1|}{}^>} 
\newcommand{\sostslx}[1]{\overline{|\nabla_x #1|}{}^>}   
\newcommand{\Argmin}[1]{{\rm Argmin}(#1)}
\newcommand{\incr}[1]{{\rm inc}\,#1}

\newcommand{\ball}[2]{{\rm B}\left(#1,#2\right)}
\newcommand{\exc}[2]{{\rm exc}\left(#1,#2\right)}
\newcommand{\haus}[2]{{\rm haus}\left(#1,#2\right)}
\newcommand{\ucalm}[2]{\overline{\rm clm}\, #1(#2)}
\newcommand{\lcalm}[2]{\underline{\rm clm}\, #1(#2)}
\newcommand{\fulcalm}[2]{{\rm clm}\, #1(#2)}
\newcommand{\Lipusc}[2]{{\rm Lipusc}\,#1(#2)}
\newcommand{\dist}[2]{{\rm dist}\left(#1,#2\right)}
\newcommand{\Lipparz}[2]{{\rm Lip_p}\,#1(#2,X)}
\newcommand{\Liploc}[2]{{\rm Lip}\,#1(#2)}

\newcommand{\supf}[2]{\varsigma({#1,#2})}

\newcommand{\Liplsc}[3]{{\rm Liplsc}\,#1(#2,#3)}
\newcommand{\calm}[3]{{\rm clm}\,#1(#2,#3)}    
\newcommand{\Lip}[3]{{\rm Lip}\,#1(#2,#3)}

\begin{document}

\title{On the quantitative solution stability of parameterized set-valued inclusions}

\author{A. Uderzo}

\address{Dipartimento di Matematica e Applicazioni, Universit\`a
di Milano-Bicocca, Via R. Cozzi, 55 - 20125 Milano, Italy}

\email{amos.uderzo@unimib.it}



\date{\today}



\begin{abstract}
The subject of the present paper are stability properties of the solution
set to set-valued inclusions. The latter are problems emerging in robust
optimization and mathematical economics, which can not be cast in
traditional generalized equations. The analysis here reported focuses
on several quantitative forms of semicontinuity for set-valued mappings,
widely investigated in variational analysis, which include, among
others, calmness. Sufficient conditions for the occurrence of
these properties in the case of the solution mapping to a parameterized
set-valued inclusion are established. Consequences on the calmness
of the optimal value function, in the context of parametric optimization,
are explored. Some specific tools for the analysis of the sufficient conditions,
in the case of set-valued inclusion with concave multifunction term, are
provided in a Banach space setting.
\end{abstract}

\maketitle


\begin{flushright}
{\small
 In fond and respectful memory of \\
 Diethard Pallaschke (1940-2020)}
\end{flushright}

\vskip1cm



\section{Introduction}

Let $F:P\times X\rightrightarrows Y$ be a given set-valued mapping
and let $C\subset Y$ be a (nonempty) closed and proper subset of $Y$,
where $P$, $X$ and $Y$ are metric spaces. Fixed any $p\in P$,
the problem:
$$
   \hbox{ find $x\in X$ such that }
  F(p,x)\subseteq C  \leqno (\mathcal{SVI}_p)
$$
is called parameterized set-valued inclusion.
As $P$ plays the role of parameter space, $(\mathcal{SVI}_p)$
is the parameterized form of a class of problems, which recently
emerged in optimization and variational analysis. More precisely,
they arise in the context of robust and vector optimization, in
mathematical economics (see \cite{BenNem98,Uder19,Uder20}), while it seems
to be reasonable that they may be of interest also in set-valued
optimization, where partial orders over sets are formalized in terms
of set inclusions (see \cite{KhTaZa15}).
Such kind of problems can not be cast in traditional generalized
equations. In the terminology of set-valued analysis, the former ones
correspond to determining the upper inverse image of a set through
a multifunction, in contrast to the latter ones, which are somehow
related to the lower inverse image.

The specific feature making a problem $(\mathcal{SVI}_p)$
essentially different from traditional generalized equations is
that the term $F$ is a multi-valued mapping, whose values must
be included in another object (set $C$). It is clear that,
whenever $F$ is single-valued, problem $(\mathcal{SVI}_p)$
can be easily embedded in the format
$$
   \hbox{ find $x\in X$ such that }
  \nullv\in f(p,x)+G(p,x), \leqno (\mathcal{GE}_p)
$$
by setting $f(p,x)=-F(p,x)$ and $G(p,x)=C$, provided that $f$
and $G$ take values in a vector space $Y$, having $\nullv$ as a null
element.
After the pioneering work of S.M. Robinson (see \cite{Robi79}),
in the last decades the literature devoted to parameterized generalized
equations mainly concentrated on the format $(\mathcal{GE}_p)$,
being guided by applications to the theory of variational
inequalities, constraint systems, optimality conditions, fixed and
coincidence points, while it has left problem $(\mathcal{SVI}_p)$ so far little explored
(see \cite{DonRoc14,KlaKum02,Mord06,Robi79} and references
therein).
It must be noted that, in the format $(\mathcal{GE}_p)$, the set-valued
term $f(p,x)+G(p,x)$ is requested to include another term, not be included in it.
Consequently, it presents a different perspective with respect to  $(\mathcal{SVI}_p)$.
Thus, the kind of relation evoked in the above formats appear hardly
compatible with each other.
It seems that the only framework able to subsume both, $(\mathcal{SVI}_p)$
and $(\mathcal{GE}_p)$, is the one considered in \cite[Section 3]{Uder17},
in connection with the so-called set-inclusion problems,  which
in its parameterized version can be formulated as
$$
   \hbox{ find $x\in X$ such that }
  F(p,x)\subseteq \Psi(p,x),  \leqno (\mathcal{SI}_p)
$$
where $\Psi:P\times X\rightrightarrows Y$ is another set-valued mapping
between metric spaces. Nevertheless, the analysis approach proposed
in \cite{Uder17} seems not be effective in the particular case of
set-valued inclusions. Indeed, it heavily relies on the assumption
that $\Psi$ is a set-covering mapping (see \cite[Definiton 2.1]{Uder17}),
which is never fulfilled if setting $\Psi$ to be a constant set-valued
mapping, i.e. $\Psi\equiv C$, as one has to do in order to embed
$(\mathcal{SVI}_p)$ in $(\mathcal{SI}_p)$.

To the best of the author's knowledge, set-valued inclusion problems
were firstly studied in \cite{Cast99}, where an error bound is
obtained by tools of convex analysis.
Subsequently, several topics related to their solution set have
started being systematically investigated in some more recent works:
solution existence and error bounds via a different approach
have been established in \cite{Uder19}; primal and dual elements for
the tangential approximation of the solution set have been provided
in \cite{Uder20}.
Following this line of research, the present paper aims at beginning
a perturbation analysis for the problem at the issue, by considering
a parameter dependence as in $(\mathcal{SVI}_p)$.
Such a perspective leads to undertake a study of the stability properties
of the solution set. A way to do this ``quantitatively" consists in investigating
Lipschitz semicontinuity properties of the solution mapping associated
with the parameterized class of set-valued inclusions, i.e. the
(generally multi-valued, with possibly empty values) mapping
$\Solv:P\rightrightarrows X$, defined by
\begin{equation}    \label{eq:Solvdef}
  \Solv(p)=\{x\in X:\ F(p,x)\subseteq C\}.
\end{equation}
These well-known properties found manifold relevant employments in
variational analysis as generalization of the notion of Lipschitz
continuity.
While investigations on Lipschitz semicontinuity properties for the
solution mapping associated to traditional generalized equations
have been already carried out (see, for instance, \cite{KlaKum02,Uder12,Uder14b}),
they appear to be new in the context of set-valued inclusions.
In the present analysis, they turn out to afford fruitful insights
on the behaviour of $\Solv$, such as nonemptiness, linear dependence
on the parameter perturbations near a reference value $\bar p\in P$
of the distance of its values from an element $\bar x\in\Solv(\bar p)$,
or from the whole set $\Solv(\bar p)$.
A key aspect of the issue is that all these features are measured
by a suitable modulus associated with each of the Lipschitz
semicontinuity properties.

The main achievements exposed in the paper are obtained by an
approach widely employed in variational analysis. According to it,
a property for set-valued mappings under investigation and an
estimate of the relative modulus are established by proving the
existence of minimizers (in fact, zeroes) of an excess function
associated to $(\mathcal{SVI}_p)$, called $\varphi_F$, a sort of merit
function measuring the violation of the inclusion in $(\mathcal{SVI}_p)$.
Conditions for the existence of those minimizers are formulated
in terms of differential constructions, which are variants of
the notion of strong slope. Such an approach thus reveals a
relation of the properties under consideration with the validity
of error bounds for proper inequalities, the latter being a topic
extensively investigated for several decades, after the seminal work
\cite{Hoff52} by A.J. Hoffman (see, among others,
\cite{AzeCor04,FaHeKrOu10,Ioff18,KrNgTh10,Krug15,Pang97,Robi75,WuYe02}).
Nevertheless, existent results on error bounds seem not be suitable
for application in the present context because of the peculiar
form of the function $\varphi_F$.

The contents of the paper are arranged in the subsequent sections
according to the following synopsis. In Section \ref{Sect:2},
several Lipschitz semicontinuity properties for set-valued mappings
acting between metric spaces, along with their moduli, are
presented, connections with other continuity and Lipschitzian type properties
of large employment in variational analysis are discussed, and
the basic tools of analysis are laid down.
In Section \ref{Sect:3}, a sufficient condition for each of the Lipschitz
semicontinuity properties presented in the previous section is established
in terms of differential conditions valid in a metric space setting.
Each of them is complemented with a quantitative estimate of the respective
modulus.
Section \ref{Sect:4} explores some consequence of the aforementioned
findings on the optimal value analysis in the context of parametric
optimization, for problems whose feasible region is defined by
a set-valued inclusion.
In Section \ref{Sect:5}, the analysis of the differential conditions
introduced in Section \ref{Sect:3} is deepened in a Banach
space setting. More precisely, under a concavity assumption on the
set-valued term defining an inclusion problem $(\mathcal{SVI}_p)$,
the fulfilment of the above conditions is shown to be guaranteed,
and somehow measured, by the occurrence of the metric $C$-increase
property.
In Section \ref{Sect:6}, conclusions and directions for expanding
the present analysis, in the light of the reported achievements,
are briefly indicated.

The notation in use throughout the paper is standard. If $A$ is a subset
of a metric space $(X,d)$, and $x\in X$, $\dist{x}{A}=\inf_{a\in A}
d(x,a)$ denotes the distance of $x$ from $A$. The closed ball
centered at $x$ with radius $r\ge 0$ is indicated with $\ball{x}{r}$,
whereas $\ball{A}{r}=\{x\in X:\ \dist{x}{A}\le r\}$ denotes the
$r$-enlargement of the set $A\subseteq X$.
Given a function $\psi:X\longrightarrow\R\cup\{\pm\infty\}$, by
$[\psi\le 0]=\psi^{-1}((-\infty,0])$ the $0$-sublevel set of $\psi$
is denoted, whereas $[\psi>0]=\psi^{-1}((0,+\infty))$ stands for
the $0$-superlevel set of $\psi$.
Given a set-valued mapping $\Phi:X\rightrightarrows Y$, $\dom\Phi=
\{x\in X:\ \Phi(x)\ne\varnothing\}$ indicates the domain of $\Phi$,
while, given $C\subseteq Y$, $\Phi^{+1}(C)=\{x\in X:\ \Phi(x)\subseteq C\}$
indicates the upper inverse (image) of $C$ through $\Phi$.
The acronyms l.s.c. and u.s.c. stand for lower semicontinuous and
upper semicontinuous, respectively.


\section{Lipschitz semicontinuity and calmness properties}   \label{Sect:2}

Before any discussion, it is proper to recall the quantitative
semicontinuity properties that will be considered in the
present work, along with their respective moduli.

\begin{definition}[Lipschitz semicontinuities]    \label{def:Lipsemcont}
Let $\Phi:X\rightrightarrows Y$ be a set-valued mapping between
metric spaces. $\Phi$ is said to be:

\begin{itemize}

\item[(i)] {\it Lipschitz lower semicontinuous} at $(\bar x,\bar y)
\in\grph\Phi$ if
there exist positive constants $\delta$ and $\ell$ such that
\begin{equation}     \label{in:defLiplsc}
  \Phi(x)\cap\ball{\bar y}{\ell d(x,\bar x)}\ne\varnothing,\quad
  \forall x\in\ball{\bar x}{\delta};
\end{equation}
the value
$$
  \Liplsc{\Phi}{\bar x}{\bar y}=\inf\{\ell>0:\ \exists\delta>0
  \hbox{ for which $(\ref{in:defLiplsc})$ holds}\}
$$
is called {\it modulus of Lipschitz lower semicontinuity} of $\Phi$
at $(\bar x,\bar y)$.

\item[(ii)] {\it calm} at $(\bar x,\bar y)\in\grph\Phi$ if there exist
positive constants $\delta$, $\zeta$ and $\ell$ such that
\begin{equation}    \label{in:defsvcalm}
   \Phi(x)\cap\ball{\bar y}{\zeta}\subseteq\ball{\Phi(\bar x)}{\ell
   d(x,\bar x)},\quad\forall x\in\ball{\bar x}{\delta};
\end{equation}
the value
$$
  \calm{\Phi}{\bar x}{\bar y}=\inf\{\ell>0:\ \exists\delta,\,\zeta>0
  \hbox{ for which $(\ref{in:defsvcalm})$ holds}\}
$$
is called {\it modulus of calmness} of $\Phi$ at $(\bar x,\bar y)$.

\item[(iii)] {\it Lipschitz upper semicontinuous}\footnote{A terminological
warning is due: Lipschitz upper semicontinuity was introduced in \cite{Robi81} under
the name of ``upper Lipschitz continuity", but later popularized as
``outer Lipschitz continuity", which is the name now prevailing in the literature.}
at $\bar x\in X$
if there exist positive constants $\delta$ and $\ell$ such that
\begin{equation}    \label{in:defLipusc}
  \Phi(x)\subseteq\ball{\Phi(\bar x)}{\ell d(x,\bar x)},\quad
  \forall x\in\ball{\bar x}{\delta};
\end{equation}
the value
$$
  \Lipusc{\Phi}{\bar x}=\inf\{\ell>0:\ \exists\delta>0
  \hbox{ for which $(\ref{in:defLipusc})$ holds}\}
$$
is called {\it modulus of Lipschitz upper semicontinuity} of $\Phi$
at $\bar x$.
\end{itemize}
\end{definition}

The next example shows that the phenomena described by the notions
in Definition \ref{def:Lipsemcont} are widely spread in nature.

\begin{example}
(i) (Bundles of linear operators) Let $(\Lin(\X,\Y),\|\cdot\|_\Lin)$
denote the Banach space of all linear bounded operators between
a normed space $(\X,\|\cdot\|)$ and a Banach space $(\Y,\|\cdot\|)$,
equipped with the operator norm $\|\cdot\|_\Lin$.
Given a nonempty subset $\mathcal{G}\subseteq\Lin(\X,\Y)$,
define a set-valued mapping $\Phi:\X\rightrightarrows\Y$ as
$$
  \Phi(x)=\{\Lambda x: \ \Lambda\in\mathcal{G}\}.
$$
Then, $\Phi$ is Lipschitz l.s.c. at $(\nullv,\nullv)\in\grph\Phi$,
with $\Liplsc{\Phi}{\nullv}{\nullv}\le\inf_{\Lambda\in\mathcal{G}}
\|\Lambda\|_\Lin$, where $\nullv$ stands for the null element in any
normed space. Indeed, fixed an arbitrary $\ell>\inf_{\Lambda\in\mathcal{G}}
\|\Lambda\|_\Lin$, one has
\begin{eqnarray*}
    \dist{\nullv}{\Phi(x)} = \inf_{\Lambda\in\mathcal{G}}\dist{\nullv}{\Lambda x}
    =\inf_{\Lambda\in\mathcal{G}}\|\Lambda x\| \le \inf_{\Lambda\in\mathcal{G}}
    \|\Lambda\|_\Lin\cdot\|x\|<\ell\|x\|,\quad\forall x\in\X,
\end{eqnarray*}
which implies
$$
  \Phi(x)\cap\ball{\nullv}{\ell\|x\|}\ne\varnothing,
  \quad\forall x\in\X.
$$

(ii) Any element $\Lambda\in\Lin(\X,\Y)$ is calm at every point $(\bar x,\Lambda\bar x)
\in\grph\Lambda$, with $\calm{\Lambda}{\bar x}{\Lambda\bar x}=\|\Lambda\|_\Lin$.

(iii) (Polyhedral finite-dimensional set-valued mappings)
A set-valued mapping $\Phi:\R^n\rightrightarrows\R^m$ is said to be
polyhedral if $\grph\Phi$ is the union of finitely many polyhedral
convex subsets of $\R^n\times\R^m$. It has been proved in \cite{Robi81}
that any polyhedral mapping $\Phi$, with $\dom\Phi=\R^n$, is Lipschitz
u.s.c. at every point of $\R^n$. A remarkable consequence of the
above result is that the solution mapping to any finite-dimensional
parameterized linear variational inequality (and, in particular, to any
complementarity problem) turns out to be Lipschitz u.s.c.
(see \cite[Exercise 3D.2]{DonRoc14}).

(iv) (Reachable state mapping) Let $(X,d)$ and $(Y,d)$ be metric spaces,
let $\Omega$ be a nonempty set and let $f:X\times\Omega\longrightarrow Y$
be a single-valued mapping. Fixed $\bar x\in X$, define
$$
  \Omega_{\bar x}=\{\omega\in\Omega:\ f(\cdot,\omega) \hbox{ is Lipschitz u.s.c. at }
  \bar x\}.
$$
Then, the set-valued mapping $\Phi:X\rightrightarrows Y$ defined by
$$
  \Phi(x)=f(x,\Omega)=\{f(x,\omega):\ \omega\in\Omega\}
$$
is Lipschitz l.s.c. at every $(\bar x,\bar y)$, with $\bar y=f(\bar x,
\bar\omega)$ such that $\bar\omega\in\Omega_{\bar x}$ ($\bar y$ exists
provided that $\Omega_{\bar x}\ne\varnothing$). Indeed, by virtue of the
Lipschitz upper semicontinuity of $f(\cdot,\bar\omega)$ at $\bar x$, there exist positive
$\kappa$ and $\delta$ such that
$$
  d(f(x,\bar\omega),f(\bar x,\bar\omega))\le\kappa d(x,\bar x),\quad
  \forall x\in\ball{\bar x}{\delta}.
$$
Therefore, one finds
$$
   f(x,\bar\omega)\in\Phi(x)\cap\ball{\bar y}{\kappa  d(x,\bar x)}
   \ne\varnothing,\quad\forall x\in\ball{\bar x}{\delta},
$$
so it is $\Liplsc{\Phi}{\bar x}{\bar y}\le\kappa$.
Notice that this example can be viewed as a generalization of
example (i).
\end{example}

For the purposes of the present analysis, the following well-known facts,
concerning the properties presented in Definition \ref{def:Lipsemcont},
are worth being mentioned. The resulting scheme should help a reader
to assess the impact of the subsequent investigations and to catch
connections with several recent lines of development within variational
analysis. For the reader's convenience, it is proper to recall that a set-valued
mapping $\Phi:X\rightrightarrows Y$ between metric spaces is said to
have the Aubin property (or to be Lipschitz-like) at $(\bar x,\bar y)
\in\grph\Phi$ if there exist positive constants $\delta$, $\zeta$ and
$\kappa$ such that
$$
   \Phi(x_1)\cap\ball{\bar y}{\zeta}\subseteq\ball{\Phi(x_2)}{\kappa
   d(x_1,x_2)},\quad\forall x_1,\, x_2\in\ball{\bar x}{\delta}.
$$

\begin{itemize}

\item {\bf Fact 1.} By elementary examples, one can show that
Lipschitz lower semicontinuity and calmness are properties
independent of each other (see, for instance,
\cite[Example 1 and 2]{Uder14b}).

\item {\bf Fact 2.} The Lipschitz upper semicontinuity of $\Phi$ at $\bar x$ implies
calmness of $\Phi$ at each point $(\bar x,y)$, with $y\in\Phi(\bar x)$, and the
inequality $\calm{\Phi}{\bar x}{y}\le\Lipusc{\Phi}{\bar x}$ holds for every
$y\in\Phi(\bar x)$, as an immediate consequence of the involved definitions
(see \cite[Chapter 3H]{DonRoc14}).

\item {\bf Fact 3.} Whenever $\Phi$ happens to be single-valued in a neighbourhood
of $\bar x$, Lipschitz lower semicontinuity at $(\bar x,\Phi(\bar x))$
and Lipschitz upper semicontinuity at $\bar x$ collapse to the same property,
postulating the existence of positive $\delta$ and $\ell$ such that
\begin{equation}     \label{def:calmsvmap}
   d(\Phi(x),\Phi(\bar x))\le\ell d(x,\bar x),\quad\forall x\in
   \ball{\bar x}{\delta}.
\end{equation}
The above property is called calmness in \cite[Chapter 1C]{DonRoc14}.
Such a choice of terminology is consistent with a certain trend in the
literature, but yields a shortcoming: the notion of calmness for
set-valued mapping as given in Definition \ref{def:Lipsemcont}(ii)
does not reduce to the property defined by inequality $(\ref{def:calmsvmap})$
in the special case of single-valued mappings (in contrast with
what stated in \cite[Chapter 3H]{DonRoc14}). Indeed, if a single-valued
mapping $\Phi:X\longrightarrow Y$ is not continuous at $\bar x$, one sees
that it is possible to choose $\zeta>0$ in such a way that $\Phi(x)
\cap\ball{\Phi(\bar x)}{\zeta}=\varnothing$, so the inclusion in
$(\ref{in:defsvcalm})$ is trivially satisfied, whereas the inequality
in $(\ref{def:calmsvmap})$ can not be.

In the special case in which $\Phi$ is a single-real-valued function,
the notion of calmness as defined in $(\ref{def:calmsvmap})$ can be split,
by considering calmness from above and calmness from below.
More precisely, $\Phi:X\longrightarrow
\R\cup\{\pm\infty\}$ is said to be {\it calm from above} at $\bar x
\in\dom\Phi$ provided that there exist positive $\delta$ and $\ell$
such that
\begin{equation}    \label{in:defucalm}
   \Phi(x)-\Phi(\bar x)\le\ell d(x,\bar x),\quad\forall
   x\in\ball{\bar x}{\delta}.
\end{equation}
The value
$$
  \ucalm{\Phi}{\bar x}=\inf\{\ell>0:\ \exists\delta>0
  \hbox{ for which $(\ref{in:defucalm})$ holds}\}
$$
is called {\it modulus of calmness from above} of $\Phi$ at $\bar x$
(see \cite[Chapter 8.F]{RocWet98}).
The version from below of calmness, along with its modulus
denoted by $\lcalm{\Phi}{\bar x}$, is defined in an analogous way.
Note that, as far as working with single-valued mappings, calmness
as defined in $(\ref{def:calmsvmap})$ implies continuity,
as well as calmness from above/below implies the corresponding (topological)
semicontinuity property at the same point.
The reader should be warned that for the aforementioned one-side versions
of calmness the terminology usage is not yet standardized.
It is worth mentioning that properties very close to calmness
emerged under various names in connection with the study of stability in
parametric optimization. More precisely, in \cite{Peno04} the term
calm is used for meaning what is called here calm from below, the term quiet
for what is called here calm from above, and the term stable for meaning
calmness in the sense of $(\ref{def:calmsvmap})$. The varying terminology
reflects the long and nonlinear history of the concepts behind, but also
the spread interest in them.

\item {\bf Fact 4.} If $\Phi:X\rightrightarrows Y$ is Lipschitz l.s.c.
at $(\bar x,\bar y)$, then its inverse mapping $\Phi^{-1}:
Y\rightrightarrows X$  is hemiregular (alias, semiregular) at $(\bar y,
\bar x)$, as understood in \cite{CiFaKr19,Krug09,MoNaWa09,Uder18}.

\item {\bf Fact 5.} If $\Phi:X\rightrightarrows Y$ is calm
at $(\bar x,\bar y)$, then its inverse mapping $\Phi^{-1}:
Y\rightrightarrows X$ is metrically subregular at $(\bar y,\bar x)$
(see \cite[Theorem 3H.3]{DonRoc14}, \cite[Proposition 2.65]{Ioff18}).

\item {\bf Fact 6.} A set-valued mapping $\Phi:X\rightrightarrows Y$
between metric spaces is called locally Lipschitz near $\bar x\in X$,
provided that there exist positive constant $\delta$ and $\ell$,
such that
\begin{eqnarray}    \label{in:Liplocdef}
  \haus{\Phi(x_1)}{\Phi(x_2)} &=& \max\{\exc{\Phi(x_1)}{\Phi(x_2)},
  \exc{\Phi(x_2)}{\Phi(x_1)}\}    \nonumber \\
  &\le& \ell d(x_1,x_2),\quad\forall  x_1,\, x_2\in\ball{\bar x}{\delta},
  \end{eqnarray}
where $\exc{A}{B}=\sup_{a\in A}\dist{a}{B}$ stands for the excess
of the set $A$ beyond the set $B$, with $A$ and $B$ being subsets
of the same metric space.
The value
$$
  \Liploc{\Phi}{\bar x}=\inf\{\ell>0:\ \exists\delta>0
  \hbox{ for which $(\ref{in:Liplocdef})$ holds}\}
$$
is called modulus of local Lipschitz continuity of $\Phi$ around
$\bar x$.
It is clear that any mapping $\Phi$, which is locally Lipschitz near
$\bar x$, is also Lipschitz u.s.c. at $\bar x$, with $\Lipusc{\Phi}{\bar x}
\le\Liploc{\Phi}{\bar x}$. This implication can not be reversed
(see Example \ref{ex:PHvsLipusc}(ii)).

Whenever inequality in $(\ref{in:Liplocdef})$
remains true with the same $\ell>0$ for every $\delta>0$, $\Phi$ is
said to be Lipschitz continuous on $X$, and the related modulus is
denoted by $\Liploc{\Phi}{X}$.

\item {\bf Fact 7.}  Calmness for set-valued mappings is always implied by
Aubin property (a.k.a. Lipschitz-likeness), which is a manifestation
of the phenomenon of metric regularity playing a fundamental role
in modern variational analysis (see \cite[Chapter 3E]{DonRoc14}).

\end{itemize}

Valuable historical comments on the genesis, the development and the successful
applications of all the notions in Definition \ref{def:Lipsemcont},
from the viewpoint of some among the major contributors to the existing
theory, can be found in \cite{BorZhu05,DonRoc14,Ioff18,KlaKum02,
Mord06,Mord18,Peno13,RocWet98}.

An aspect that makes impossible any direct comparison of the
properties in Definition \ref{def:Lipsemcont}(i) and (ii) with
purely topological semicontinuity properties and with the properties
based on the Painlev\'e-Kuratowski convergence (see \cite[Chapter 3A]{DonRoc14}),
when revisited in metric spaces, is the reference to a point of the graph,
instead of to a point in the definition space.
The next example reveals that the notion in Definition \ref{def:Lipsemcont}(iii)
and the Pompeiu-Hausdorff continuity are independent of each other,
as one expects.
According to \cite[Chapter 3B]{DonRoc14} a set-valued mapping between
metric spaces $\Phi:X\rightrightarrows Y$ is said to be Pompeiu-Hausdorff
continuous at $\bar x\in\dom\Phi$ if
$$
  \lim_{x\to\bar x}\haus{\Phi(x)}{\Phi(\bar x)}=0.
$$

\begin{example}     \label{ex:PHvsLipusc}
(i) (Pompeiu-Hausdorff continuity without Lipschitz upper semicontinuity)
Let $X=Y=\R$ be equipped with its Euclidean metric structure. Define
a set-valued mapping $\Phi:\R\rightrightarrows\R$ as
$$
  \Phi(x)=[-\sqrt{|x|},\sqrt{|x|}],
$$
and take $\bar x=0$. Since it is $\Phi(0)=\{0\}$, so $\Phi(0)
\subseteq \Phi(x)$ for every $x\in\R$, and hence
$$
  \exc{\Phi(0)}{\Phi(x)}=0,\quad\forall x\in\R,
$$
whereas it is
$$
  \exc{\Phi(x)}{\Phi(0)}=\sqrt{|x|},\quad\forall x\in\R,
$$
clearly it holds $\lim_{x\to 0}\haus{\Phi(x)}{\Phi(0)}=0$.
Nevertheless, since for any $\ell>0$ there is no $\delta>0$
such that the inequality
$$
  \sqrt{|x|}\le\ell|x|, \quad\forall x\in\ball{0}{\delta}
$$
takes place, the inclusion
$$
   \Phi(x)\subseteq\ball{\{0\}}{\ell|x|}=[-\ell|x|,\ell|x|],
   \quad\forall x\in\ball{0}{\delta},
$$
fails to be true.

(ii) (Lipschitz upper semicontinuity without Pompeiu-Hausdorff continuity
and local Lipschitz continuity)
Let $X$ and $Y$ be as in the previous case. Define a set-valued mapping
$\Phi:\R\rightrightarrows\R$ by setting
$$
   \Phi(x)=\left\{\begin{array}{ll}
   (-\infty,0], & \hbox{ if } x<0,   \\
   \R, & \hbox{ if } x=0,   \\
   {}[0,+\infty) &  \hbox{ if } x>0,
   \end{array}  \right.
$$
and take $\bar x=0$. Since $\ball{\Phi(0)}{\ell|x|}=\R$ for every $x\in\R$
and $\ell>0$, the inclusion
$$
  \Phi(x)\subseteq\ball{\Phi(0)}{\ell|x|}=\R,\quad\forall x\in\ball{0}{\delta}
$$
is evidently satisfied for every $\ell,\, \delta>0$. Nonetheless, since it is
$\exc{\Phi(0)}{\Phi(x)}=+\infty$ for  every $x\in\R\backslash\{0\}$,
one has $\lim_{x\to 0}\haus{\Phi(x)}{\Phi(0)}=+\infty\ne 0$.

Since, whenever $x_1,\, x_2\in\R$ are such that $x_1<0<x_2$, one finds
$$
  \haus{\Phi(x_1)}{\Phi(x_2)}\ge\sup_{y\in (-\infty,0]}\dist{y}{[0,+\infty)}
  =+\infty,
$$
note that $\Phi$ fails also to be locally Lipschitz near $0$.
\end{example}


The analysis of the properties recalled in Definition \ref{def:Lipsemcont}
will be performed in the case of the set-valued mapping $\Solv:P\rightrightarrows X$,
describing the solution stability of problems $(\mathcal{SVI}_p)$
under parameter perturbation.
Since all of them have a purely metric nature, it seems to be
natural, at a first step, to approach this question ``iuxta
propria principia", namely in a metric space setting. A way
to accomplish such a task relies on the excess function
that can be associated to a problem $(\mathcal{SVI}_p)$, i.e.
the functional $\varphi_F:X\longrightarrow\R\cup\{\pm\infty\}$,
defined as
$$
   \varphi_F(p,x)=\exc{F(p,x)}{C}.
$$
Such an approach enables one to characterize the solution of
a $(\mathcal{SVI}_p)$ problem as a zero of $\varphi_F(p,\cdot)$, according
to a successful strategy in addressing different variational
problems, even in their metric space formulation. In fact,
modern variational analysis offers a well developed apparatus
of tools and techniques for carrying out a quantitative study of
the stability properties of zeroes of functionals.
The remaining part of the present section is devoted to gather
all those elements of set-valued and variational analysis,
which are needed to implement this approach in the subsequent
sections.
Let us start with a lemma that links semicontinuity properties
of a set-valued mapping $\Phi:X\rightrightarrows Y$ with those
of the related excess function $\varphi_\Phi:X\longrightarrow
\R\cup\{\pm\infty\}$.

\begin{lemma}     \label{lem:Phivarphi}
Let $\Phi:X\rightrightarrows Y$ be a set-valued mapping between
metric spaces and let $C\subseteq Y$ be a nonempty closed set.
Define $\varphi_\Phi:X\longrightarrow\R\cup\{\pm\infty\}$ as
$\varphi_\Phi(x)=\exc{\Phi(x)}{C}$.

(i) If $\Phi$ is l.s.c. at $\bar x$, then $\varphi_\Phi$ is l.s.c.
at the same point.

(ii) If $\Phi$ is Lipschitz u.s.c. at $\bar x$, then $\varphi_\Phi$
is calm from above at the same point and
\begin{equation}   \label{in:calmLipuscest}
  \ucalm{\varphi_\Phi}{\bar x}\le\Lipusc{\Phi}{\bar x}.
\end{equation}
In particular, if $\Phi$ is Lipschitz u.s.c. at $\bar x$, then
$\varphi_\Phi$ is u.s.c. at $\bar x$.
\end{lemma}

\begin{proof}
(i) The proof is given in full detail in \cite[Lemma 2.3]{Uder19}, upon the
assumptions that $Y$ is a normed space and $C$ is a closed, convex
cone of it. A perusal of the argument employed there reveals that
neither the convexity of $C$ nor the linear structure of $C$
are actually exploited in the proof, relying instead on basic
definitions and inequalities valid in any metric space.

(ii) Take an arbitrary $\ell>\Lipusc{\Phi}{\bar x}$. Then, there
exists $\delta>0$ such that inclusion $(\ref{in:defLipusc})$
holds. Consequently, by using the triangular inequality for
the excess, one obtains
\begin{eqnarray*}
  \varphi_\Phi(x)=\exc{\Phi(x)}{C}\le\exc{\Phi(x)}{\Phi(\bar x)}+
  \exc{\Phi(\bar x)}{C}\le\ell d(x,\bar x)+\varphi_\Phi(\bar x),
  \quad\forall x\in\ball{\bar x}{\delta}.
\end{eqnarray*}
This shows that $\varphi_\Phi$ is calm from above at $\bar x$ and
that $\ucalm{\varphi_\Phi}{\bar x}\le\ell$. The estimate in
$(\ref{in:calmLipuscest})$ is true because of the arbitrariness
of $\ell>\Lipusc{\Phi}{\bar x}$.

The last statement in the thesis is a consequence of fact that,
as remarked in Fact 3,
calmness from above for a single-real-valued function implies
upper semicontinuity. Indeed, passing to the $\limsup$ in the
above inequality, one finds
$$
  \limsup_{x\to\bar x}\varphi_\Phi(x)\le\limsup_{x\to\bar x}\,
  [\ell d(x,\bar x)+\varphi_\Phi(\bar x)]=\varphi_\Phi(\bar x).
$$
This completes the proof.
\end{proof}

\begin{remark}     \label{rem:varphirealval}
In several circumstances, it will be proper to work with
an excess function $\varphi_\Phi$, which is real valued on $X$. It is readily
seen that this happens if the two following assumptions are
taken on $\Phi$:

(i) $\dom\Phi=X$ (so that $\varphi_\Phi(X)\subseteq[0,+\infty]$);

(ii) the set-valued mapping $\Phi$ is {\it bounded-valued away from}
the set $C$, meaning that for every $x\in X$ the set $\Phi(x)
\backslash C$ is bounded as a subset of a metric space (so that
$\varphi_\Phi(x)<+\infty$, for every $x\in X$).
\end{remark}

In order to formalize the differential conditions, upon which Lipschitz
semicontinuity properties of $\Solv$ will be established, the notion of
strict outer slope plays a crucial role.
Such tool of analysis in metric spaces can be presented as a regularization
of the well-known notion of (strong) slope, which was introduced in \cite{DeMaTo80}.
Given a function $\psi:X\longrightarrow\R\cup\{\pm\infty\}$,
defined on a metric space $(X,d)$, and $\bar x\in\psi^{-1}(\R)$, the
real-extended value
$$
   \stsl{\psi}(\bar x)=\left\{ \begin{array}{ll}
                       0, & \qquad\text{if $\bar x$ is a local minimizer for $\psi$},  \\
                       \displaystyle\limsup_{x\to\bar x}
                       \frac{\psi(\bar x)-\psi(x)}{d(x,\bar x)}, &
                       \qquad\text{otherwise,}
                 \end{array}
   \right.
$$
is called ({\it strong}) {\it slope} of $\psi$ at $\bar x$.
The real-extended value
\begin{eqnarray*}
    \sostsl{\psi}(\bar x) &=& \lim_{\epsilon\to 0^+}\inf
    \{\stsl{\psi}(x): x\in \ball{\bar x}{\epsilon},\ \psi(\bar x)<
     \psi(x)<\psi(\bar x)+\epsilon\}   \\
      &=& \liminf_{x\to\bar x\atop \psi(x)\downarrow\psi(\bar x)}
      \stsl{\psi}(x)
\end{eqnarray*}
is called {\it strict outer slope} of $\psi$ at $\bar x$.

\begin{example}    \label{ex:stsl}
(i) It is well known that, whenever a function $\psi:\X\longrightarrow\R
\cup\{\pm\infty\}$, defined on a Banach space, is Fr\'echet differentiable
at $x\in\psi^{-1}(\R)$, with Fr\'echet derivative $\Fder\psi(x)$, then
its strong slope at $x$ can be readily calculated as $\stsl{\psi}(x)=
\|\Fder\psi(x)\|$ (see, for instance, \cite[Chapter 1.2]{Ioff00}).

(ii) In view of the analysis exposed in Section \ref{Sect:5}, it is
useful to mention that for a function $\psi:\X\longrightarrow\R
\cup\{\pm\infty\}$ defined in a Banach space, which is l.s.c. and
convex, it holds
$$
    \stsl{\psi}(x)=\dist{\nullv^*}{\partial\psi(x)},
    \quad\forall x\in\psi^{-1}(\R),
$$
where $\partial\psi(x)$ denotes the subdifferential of $\psi$ at
$x$, in the sense of convex analysis, and $\nullv^*$ stands for the null
element of the dual space $\X^*$ to $\X$. In such case, the value
$\dist{\nullv^*}{\partial\psi(x)}$ is often referred to as the
subdifferential slope of $\psi$ at $x$ (see, for instance,
\cite[Theorem 5]{FaHeKrOu10}).
\end{example}

For further reading on the theme of strong slope and its several
variations the reader is refereed to \cite[Chapter 3.1.2]{Ioff18},
\cite{FaHeKrOu10} and \cite[Chapter 1.6]{Peno13}.

\vskip1cm


\section{Conditions for Lipschitz semicontinuity of $\Solv$}   \label{Sect:3}

Throughout the current section, the equality $\dom F=P\times X$ will be
kept in force as a standing assumption on the set-valued mapping $F$,
which will be supposed also to take closed values.
Before discussing the Lipschitz semicontinuity properties of $\Solv$,
it is worthwhile to come back to the equality
\begin{equation}    \label{eq:levsetupinvchar}
  \Solv(p)=[\varphi_F(p,\cdot)\le 0]=F(p,\cdot)^{+1}(C).
\end{equation}
A topological consequence of it is that, whenever the set-valued mapping $F(p,\cdot):
X\rightrightarrows Y$ is supposed to be l.s.c. on $X$ for
every fixed $p\in P$, the set $\Solv(p)=X\backslash[\varphi_F(p,\cdot)>0]$
is closed (possibly, empty), by virtue of the lower
semicontinuity of the functional $\varphi_F(p,\cdot)$
(remember Lemma \ref{lem:Phivarphi}(i)). Thus, upon the
above condition, $\Solv$ turns out to be closed-valued.

In the statement of the next result, the following partial version of
the strict outer slope will be employed:
\begin{eqnarray}    \label{def:sostslx}
  \sostslx{\varphi_F}(\bar p,\bar x) &=&\lim_{\epsilon\to 0^+}\inf
    \{\stsl{\varphi_F(p,\cdot)}(x):\ (p,x)\in \ball{\bar p}{\epsilon}\times
    \ball{\bar x}{\epsilon}, \\  \nonumber
    & &\hskip4.5cm \varphi_F(\bar p,\bar x)<\varphi_F(p,x)<\varphi_F
    (\bar p,\bar x)+\epsilon\} \\  \nonumber
    &=&  \liminf_{(p,x)\to(\bar p,\bar x)\atop \varphi_F(p,x)\downarrow \varphi_F(\bar p,\bar x)}
    \stsl{\varphi_F(p,\cdot)}(x).
\end{eqnarray}

\begin{theorem}[Lipschitz lower semicontinuity of $\Solv$]    \label{thm:LiplscSolv}
With reference to a parameterized set-valued inclusion $(\mathcal{SVI}_p)$,
let $\bar p\in P$ and let $\bar x\in\Solv(\bar p)$. Suppose that:

(i) $(X,d)$ is metrically complete;

(ii) the mapping $F(\cdot,\bar x):P\rightrightarrows Y$ is
Lipschitz u.s.c. at $\bar p$;

(iii) there exists $\delta>0$ such that, for every $p\in\ball{\bar p}{\delta}$,
each mapping $F(p,\cdot):X\rightrightarrows Y$ is l.s.c. on $X$;

(iv) it holds $\sostslx{\varphi_F}(\bar p,\bar x)>0$.

\noindent Then, the solution mapping $\Solv:P\rightrightarrows X$
is Lipschitz l.s.c. at $(\bar p,\bar x)$ and the following estimate holds
\begin{equation}    \label{in:LiplscSolv}
  \Liplsc{\Solv}{\bar p}{\bar x}\le{\Lipusc{F(\cdot,\bar x)}{\bar p}
  \over  \sostslx{\varphi_F}(\bar p,\bar x)}.
\end{equation}

\end{theorem}

\begin{proof}
By hypothesis (iv), it is possible to fix the value of $\sigma$ in such
a way that
$$
   \sostslx{\varphi_F}(\bar p,\bar x)>\sigma>0.
$$
According to the definition of strict outer slope,
the above inequality means that there exists $\eta>0$ such that
for every $\epsilon\in (0,\eta)$ it holds
\begin{equation}     \label{in:stslpos}
    \stsl{\varphi_F(p,\cdot)}(x)>\sigma,\quad
    \forall (p,x)\in\ball{\bar p}{\epsilon}\times\ball{\bar x}{\epsilon}:\
    0<\varphi_F(p,x)<\epsilon.
\end{equation}
Without any loss of generality, it is possible to assume that
$\eta<\delta$, where $\delta$ is as in hypothesis (iii).
In the light of Lemma \ref{lem:Phivarphi}(ii), from the hypothesis (ii)
one deduces that the function $p\mapsto\varphi_F(p,\bar x)$ is
calm from above at $\bar p$, with $\ucalm{\varphi_F(\cdot,\bar x)}{\bar p}
\le\Lipusc{F(\cdot,\bar x)}{\bar p}$. Thus, taken an arbitrary
$\ell>\Lipusc{F(\cdot,\bar x)}{\bar p}$, there exists $\delta_\ell>0$
such that
\begin{equation}     \label{in:phibarxval}
  \varphi_F(p,\bar x)\le\varphi_F(\bar p,\bar x)+\ell d(p,\bar p)
  =\ell d(p,\bar p),\quad\forall p\in\ball{\bar p}{\delta_\ell}.
\end{equation}
Notice that, by reducing its value if needed, it is possible to
pick $\delta_\ell$ in such a way that
\begin{equation}    \label{in:deltal}
   \delta_\ell<\min\left\{{\sigma\eta\over
   2(\ell+1)},\, {\eta\over 2(\ell+1)}\right\}.
\end{equation}
With such a choice, one has, in particular, $\delta_\ell<{\eta\over 2}$.
Now, fix an arbitrary $p\in\ball{\bar p}{\delta_\ell}\backslash
\{\bar p\}$ and consider the function $\varphi_F(p,\cdot):X\longrightarrow
[0,+\infty]$. Since it is $\delta_\ell<\eta<\delta$, then by virtue of Lemma \ref{lem:Phivarphi}(i)
and the hypothesis (iii), this function is l.s.c. on $X$. It is clearly
bounded from below. Moreover, by taking into account inequality (\ref{in:phibarxval}),
one has
$$
  \varphi_F(p,\bar x)\le\inf_{x\in X}\varphi_F(p,x)+\ell d(p,\bar p).
$$
By the Ekeland variational principle, corresponding to
$$
  \lambda={\ell d(p,\bar p)\over\sigma},
$$
there exists $x_\lambda\in X$ with the following properties:
\begin{equation}   \label{in:EVP1}
    \varphi_F(p,x_\lambda)\le\varphi_F(p,\bar x)\le\ell d(p,\bar p),
\end{equation}
\begin{equation}   \label{in:EVP2}
   d(x_\lambda,\bar x)\le\lambda,
\end{equation}
and
\begin{equation}   \label{in:EVP3}
    \varphi_F(p,x_\lambda)<\varphi_F(p,x)+\sigma d(x,x_\lambda),
    \quad\forall x\in X\backslash\{x_\lambda\}.
\end{equation}
Under the current assumptions, the above properties allow one
to show that it is $\varphi_F(p,x_\lambda)=0$.
Indeed, suppose, ab absurdo, that it is $\varphi_F(p,x_\lambda)>0$.
As a consequence of inequality $(\ref{in:EVP3})$, one has
$$
  {\varphi_F(p,x_\lambda)-\varphi_F(p,x)\over d(x,x_\lambda)}<\sigma,
  \quad\forall x\in X\backslash\{x_\lambda\},
$$
whence one obtains
\begin{equation}    \label{in:stslcontr}
  \stsl{\varphi_F(p,\cdot)}(x_\lambda)=\lim_{r\to 0^+}
  \sup_{x\in\ball{x_\lambda}{r}\backslash\{x_\lambda\}}
  {\varphi_F(p,x_\lambda)-\varphi_F(p,x)\over d(x,x_\lambda)}
    \le\sigma.
\end{equation}
On the other hand, by recalling inequality $(\ref{in:deltal})$,
one sees that
\begin{equation}   \label{in:pcontr}
   d(p,\bar p)\le\delta_\ell<{\eta\over 2},
\end{equation}
and, on account of inequality $(\ref{in:EVP2})$,
\begin{equation}   \label{in:xcontr}
  d(x_\lambda,\bar x)\le{\ell d(p,\bar p)\over\sigma}
  <{\ell\over\sigma}\cdot{\sigma\eta\over
   2(\ell+1)}<{\eta\over 2}.
\end{equation}
Besides, because of inequalities $(\ref{in:deltal})$ and $(\ref{in:EVP1})$,
it is true that
\begin{equation}    \label{in:varphicontr}
  \varphi_F(p,x_\lambda)\le\ell\delta_\ell<\ell{\eta\over
   2(\ell+1)}<{\eta\over 2}.
\end{equation}
Inequality $(\ref{in:stslcontr})$, along with inequalities $(\ref{in:pcontr})$,
$(\ref{in:xcontr})$, and $(\ref{in:varphicontr})$, contradicts $(\ref{in:stslpos})$
if taking $\epsilon=\eta/2$.
The above argument proves that it is actually $\varphi_F(p,x_\lambda)=0$,
which means that $F(p,x_\lambda)\subseteq C$, and hence
$x_\lambda\in\Solv(p)$.
Since it is $d(x_\lambda,\bar x)\le{\ell d(p,\bar p)\over\sigma}$,
it results in
$$
   \Solv(p)\cap\ball{\bar x}{{\ell d(p,\bar p)\over\sigma}}
   \ne\varnothing.
$$
By arbitrariness of $p\in\ball{\bar p}{\delta_\ell}\backslash
\{\bar p\}$, the last relation amounts to say that $\Solv$ is Lipschitz
l.s.c. at $(\bar p,\bar x)$, and it holds
$$
   \Liplsc{\Solv}{\bar p}{\bar x}\le{\ell\over\sigma}.
$$
The arbitrariness of $\ell>\Lipusc{F(\cdot,\bar x)}{\bar p}$ and
of $\sigma<\sostslx{\varphi_F}(\bar p,\bar x)$ enables one to achieve
the estimate in the thesis, thereby completing the proof.
\end{proof}

The reader should notice that Theorem \ref{thm:LiplscSolv} provides
a condition for the local solvability of problems $(\mathcal{SVI}_p)$
under parameter perturbation. Furthermore, through the estimate
$(\ref{in:LiplscSolv})$, it affords quantitative information on the
stability of the solution mapping. Unfortunately, as happens for many
implicit function theorems, the differential condition upon which
it can be established (hypothesis (iv) in Theorem \ref{thm:LiplscSolv})
is only sufficient. This fact is illustrated by the next example.

\begin{example}     \label{ex:difcondnotsuff}
Let $P=X=Y=\R$ be endowed with its usual Euclidean metric structure,
let $C=[0,+\infty)$ and let $F:\R\times\R\rightrightarrows\R$ be defined
by
$$
  F(p,x)=[p^3+x^3,+\infty).
$$
Since $F(p,x)\subseteq C$ iff $p^3+x^3=(p+x)(p^2-px+x^2)\ge 0$, in this
case the solution mapping $\Solv:\R\rightrightarrows\R$ associated
with the inclusion problem takes the simple form
$$
  \Solv(p)=[-p,+\infty),\quad\forall p\in\R.
$$
In particular, one sees that, letting $\bar p=\bar x=0$, it is
$0\in\Solv(0)$.
Since $F$ is the epigraphical set-valued mapping related to the continuous
function $(p,x)\mapsto p^3+x^3$, it is l.s.c. on $\R\times\R$. As a
consequence, there exists $\delta>0$ such that each set-valued mapping
$x\leadsto F(p,x)$ is l.s.c. on $\R$, for every $p\in\ball{0}{\delta}$.
To check the Lipschitz lower semicontinuity of the set-valued mapping
$p\leadsto F(p,0)=[p^3,+\infty)$ mentioned in the hypothesis (ii) of
Theorem \ref{thm:LiplscSolv}, observe that
$$
  F(p,0)\subseteq F(0,0)=[0,+\infty),\quad\forall p\ge 0.
$$
For any $p\in (0,1)$, as it is $p^3>-|p|$, one finds
$$
  F(p,0)=[p^3,+\infty)\subseteq [-|p|,+\infty)=
  \ball{F(0,0)}{|p|}.
$$
Thus, the set-valued mapping $p\leadsto F(p,0)=[p^3,+\infty)$
is Lipschitz u.s.c. at $0$, with $\Lipusc{F(\cdot,0)}{0}\le 1$.
It is readily seen that the function $\varphi_F:\R\times\R\longrightarrow\R$
is given in the present circumstance by
\begin{eqnarray*}
  \varphi_F(p,x) &=& \exc{[p^3+x^3,+\infty)}{[0,+\infty)} \\
  &=&  \left\{ \begin{array}{ll}
                       0, & \qquad\hbox{ if } (p,x)\in\R\times\R:\ p+x\ge 0,  \\
                       -(p^3+x^3), &   \qquad\hbox{ if } (p,x)\in\R\times\R:\
                       p+x<0.
                 \end{array}
   \right.
\end{eqnarray*}
Recalling the strong slope estimate remarked in Example \ref{ex:stsl}(i),
one has
$$
  \stsl{\varphi_F(p,\cdot)}(x)=\left|{\partial\over\partial x}\varphi_F(p,x)\right|=
  \left\{ \begin{array}{ll}
                       0, & \qquad\hbox{ if } (p,x)\in\R\times\R:\ p+x> 0,  \\
                       3x^2, &   \qquad\hbox{ if } (p,x)\in\R\times\R:\
                       p+x<0.
                 \end{array}
   \right.
$$
Consequently, according to the definition of strict outer slope,
one finds
$$
  \sostslx{\varphi_F}(0,0)=0,
$$
because points $(p,x)$ such that $p+x<0$ and $x=0$ (hence such that
$\stsl{\varphi_F}(p,x)=0$) can be found in each set $((-\epsilon,\epsilon)
\times(-\epsilon,\epsilon))\cap\{(p,x)\in\R\times\R:\ 0<-(p^3+x^3)<
\epsilon\}$, with $\epsilon>0$. This shows that hypothesis (iv)
of Theorem \ref{thm:LiplscSolv} is not satisfied.
Nonetheless, the mapping $\Solv$ turns out to be Lipschitz l.s.c.
at $(0,0)$. Indeed, directly from the expression of $\Solv$,
one sees that
$$
   \Solv(p)\cap\ball{0}{|p|}\ne\varnothing,\quad\forall p\in\R,
$$
so $\Liplsc{\Solv}{0}{0}\le 1$.
\end{example}


\begin{theorem}[Calmness of $\Solv$]    \label{thm:calmSolv}
With reference to a parameterized set-valued inclusion $(\mathcal{SVI}_p)$,
let $\bar p\in P$ and let $\bar x\in\Solv(\bar p)$. Suppose that:

(i) $(X,d)$ is metrically complete;

(ii) the set-valued mapping $F(\bar p,\cdot):X\rightrightarrows Y$
is l.s.c. on $X$;

(iii) the mapping $F$ is locally Lipschitz near $(\bar p,\bar x)$;

(iv) it holds $\sostsl{\varphi_F(\bar p,\cdot)}(\bar x)>0$.

\noindent Then, the solution mapping $\Solv:P\rightrightarrows X$
is calm at $(\bar p,\bar x)$ and the following estimate holds
$$
  \calm{\Solv}{\bar p}{\bar x}\le{\Lip{F}{\bar p}{\bar x}
  \over\sostsl{\varphi_F(\bar p,\cdot)}(\bar x)}.
$$

\end{theorem}

\begin{proof}
According to hypothesis (iii), taken an arbitrary $\ell>\Lip{F}{\bar p}{\bar x}$
there must exist $\delta>0$ such that
\begin{eqnarray}      \label{in:locLipF}
   \haus{F(p_1,x_1)}{F(p_2,x_2)} &\le &\ell\max\{d(p_1,p_2),\, d(x_1,x_2)\},\\
   & &\hskip1cm\forall (p_1,x_1),\, (p_2,x_2)\in
   \ball{\bar p}{\delta}\times \ball{\bar x}{\delta}.
   \nonumber
\end{eqnarray}
According to hypothesis (iv), taken any $\sigma>0$ such that
$$
  \sostsl{\varphi_F(\bar p,\cdot)}(\bar x)>\sigma>0,
$$
there exists $\eta>0$ such that, for every $\epsilon\in (0,\eta)$,
it holds
\begin{equation}    \label{in:contrstsldeg}
   \stsl{\varphi_F(\bar p,\cdot)}(x)>\sigma,\quad\forall x\in
   \ball{\bar x}{\epsilon},\ 0<\varphi_F(\bar p,x)<\epsilon.
\end{equation}
Let us take positive reals $\delta_*$ and $\zeta$ in such a way that
\begin{equation}    \label{in:zetadeltastarc}
   \delta_*<\min\left\{\delta,\, {\eta\over 2(\ell+1)},\, {\eta\sigma\over 4(\ell+1)}\right\}
   \quad\hbox{ and }\quad
   \zeta <\min\left\{\delta,\, {\eta\over 4}\right\}.
\end{equation}
Now, fix an arbitrary $p\in\ball{\bar p}{\delta_*}\backslash
\{\bar p\}$. If $\Solv(p)\cap\ball{\bar x}{\zeta}=\varnothing$,
then nothing is left to prove. Otherwise, take an arbitrary
$x_p\in\Solv(p)\cap\ball{\bar x}{\zeta}$.
Consider the function $\varphi(\bar p,\cdot):X\longrightarrow [0,+\infty]$.
By virtue of hypothesis (ii), Lemma \ref{lem:Phivarphi}(i)
ensures that this function is l.s.c. on $X$. It is clearly
bounded from below.
Since $d(x_p,\bar x)\le\zeta<\delta$ and $d(p,\bar p)\le\delta_*<\delta$,
from inequality $(\ref{in:locLipF})$ and the triangular inequality
for the excess, it follows
\begin{eqnarray*}
  \varphi_F(\bar p,x_p)=\exc{F(\bar p,x_p)}{C}\le
  \exc{F(\bar p,x_p)}{F(p,x_p)}+\exc{F(p,x_p)}{C}\le
  \ell d(p,\bar p).
\end{eqnarray*}
Thus, it is
$$
  \varphi_F(\bar p,x_p)\le\inf_{x\in X}\varphi_F(\bar p,x)+\ell d(p,\bar p)
  =\ell d(p,\bar p).
$$
By applying the Ekeland variational principle, one can assert that,
corresponding to the value
$$
   \lambda={\ell d(p,\bar p)\over\sigma},
$$
there exists $x_\lambda\in X$, satisfying the below properties:
\begin{equation}      \label{in:EVP1calm}
   \varphi_F(\bar p,x_\lambda)\le \varphi_F(\bar p,x_p),
\end{equation}
\begin{equation}      \label{in:EVP2calm}
  d(x_\lambda,x_p)\le\lambda,
\end{equation}
and
\begin{equation}     \label{in:EVP3calm}
  \varphi_F(\bar p,x_\lambda)<\varphi_F(\bar p,x)+\sigma d(x,x_\lambda),
  \quad\forall x\in X\backslash\{x_\lambda\}.
\end{equation}
The inequality $(\ref{in:EVP3calm})$ implies
$$
  {\varphi_F(\bar p,x_\lambda)-\varphi_F(\bar p,x)\over
  d(x,x_\lambda)}<\sigma,\quad\forall x\in X\backslash\{x_\lambda\},
$$
wherefrom it follows
$$
  \stsl{\varphi_F(\bar p,\cdot)}(x_\lambda)\le\sigma.
$$
The last inequality entails that it must be $\varphi_F(\bar p,x_\lambda)
=0$. Indeed, if supposing $\varphi_F(\bar p,x_\lambda)>0$, by recalling
the choice of $\delta_*$ and $\zeta$ in $(\ref{in:zetadeltastarc})$,
along with inequality $(\ref{in:EVP1calm})$, one obtains
$$
  \varphi_F(\bar p,x_\lambda)\le\ell d(p,\bar p)<{\eta\over 2}
$$
and
$$
  d(x_\lambda,\bar x)\le d(x_\lambda,x_p)+d(x_p,\bar x)\le\lambda+\zeta
  \le {\ell d(p,\bar p)\over\sigma}+\zeta< {\eta\over 4}+{\eta\over 4}
  ={\eta\over 2}.
$$
Thus, inequality $(\ref{in:contrstsldeg})$ turns out to be contradicted
with $\epsilon=\eta/2$. The fact that it is $\varphi_F(\bar p,x_\lambda)=0$
implies $x_\lambda\in\Solv(\bar p)$. By taking into account inequality
$(\ref{in:EVP2calm})$, one finds
$$
  \dist{x_p}{\Solv(\bar p)}\le d(x_p,x_\lambda)\le
  {\ell d(p,\bar p)\over\sigma},
$$
which amounts to say that $x_p\in\ball{\Solv(\bar p)}{\ell d(p,\bar p)/\sigma}$.
By arbitrariness of $x_p\in\Solv(p)\cap\ball{\bar x}{\zeta}$ and
$\bar p\in\ball{\bar p}{\delta_*}\backslash\{\bar p\}$, the above argument
shows that $\Solv$ is calm at $(\bar p,\bar x)$, with $\calm{\Solv}{\bar p}{\bar x}\le
\ell/\sigma$. The arbitrariness of $\ell>\Lip{F}{\bar p}{\bar x}$ and
of $\sigma<\sostsl{\varphi_F(\bar p,\cdot)}(\bar x)$ leads to achieve
the estimate in the thesis, thereby completing the proof.
\end{proof}

In the same vein of Example \ref{ex:difcondnotsuff}, the next counterexample
shows that condition (iv) in Theorem \ref{thm:calmSolv} is far from
being necessary.

\begin{example}
Let us consider the same inclusion problem introduced in Example \ref{ex:difcondnotsuff}.
As a continuously differentiable function, $(p,x)\mapsto p^3+x^3$ is strictly
differentiable at $(\bar p,\bar x)=(0,0)$, and hence locally Lipschitz near
that point. Consequently, $F$ is locally Lipschitz near $(0,0)$. Since it is
\begin{eqnarray*}
  \varphi_F(0,x) = \left\{ \begin{array}{ll}
                       0, & \qquad\hbox{ if } x\ge 0,  \\
                       -x^3, &   \qquad\hbox{ if } x<0,
                 \end{array}
   \right.
\end{eqnarray*}
it results in
\begin{eqnarray*}
    \sostsl{\varphi_F(0,\cdot)}(0)& =& \lim_{\epsilon\to 0^+}\inf
    \{\stsl{\varphi_F(0,\cdot)}(x): x\in [-\epsilon,\epsilon],\
    0<\varphi_F(0,x)<\epsilon\} \\
    & =& \lim_{\epsilon\to 0^+}\inf\{3x^2: x\in [-\epsilon,0)\}=0.
\end{eqnarray*}
Thus, hypothesis (iv) of Theorem \ref{thm:calmSolv} fails to be
fulfilled. In spite of such a failure, the solution mapping $\Solv:
\R\rightrightarrows\R$ associated with the inclusion problem is not only
calm at $(0,0)$, but even Lipschitz continuous on $\R$.
\end{example}

As a further comment to the so far exposed results, it is to be noted
that the differential conditions appearing in Theorem \ref{thm:LiplscSolv}
and in Theorem \ref{thm:calmSolv} are different. Indeed, whereas both
of them are built by means of the partial strong slope with respect to
the variable $x$, the condition (iv) in Theorem \ref{thm:LiplscSolv}
considers the $\liminf_{\epsilon\to 0^+}$ regularization, with both
$p$ and $x$ varying near $\bar p$ and $\bar x$, respectively. In
contrast to this, condition (iv) in Theorem \ref{thm:calmSolv} requires only
$x$ to vary near $\bar x$, while it is kept $p=\bar p$.


In view of the formulation of the last result of this section, given
$F:P\times X\rightrightarrows Y$ and $\bar p\in P$, let us define
the values
$$
  \tau_{\bar p}=\inf\{\stsl{\varphi_F(\bar p,\cdot)}(x):\ x\in X
  \backslash\Solv(\bar p)\}
$$
and
\begin{eqnarray*}
  \Lipparz{F}{\bar p} &=& \inf\{\ell>0:\ \exists\delta>0:\ \sup_{x\in X}
  \haus{F(p_1,x)}{F(p_1,x)}\le\ell d(p_1,p_2), \\
  & &\forall p_1,\, p_2
  \in\ball{\bar p}{\delta}\}.
\end{eqnarray*}

\begin{theorem}[Lipschitz upper semicontinuity of $\Solv$]    \label{thm:LipuscSolv}
With reference to a parameterized set-valued inclusion $(\mathcal{SVI}_p)$,
let $\bar p\in P$, with $\Solv(\bar p)\ne\varnothing$. Suppose that:

(i) $(X,d)$ is metrically complete;

(ii) the set-valued mapping $F(\bar p,\cdot):X\rightrightarrows Y$
is l.s.c. on $X$;

(iii) the mapping $F$ is locally Lipschitz near $\bar p$ with respect to $p$,
uniformly in $x\in X$;

(iv) it is $\tau_{\bar p}>0$.

\noindent Then, the solution mapping $\Solv:P\rightrightarrows X$
is Lipschitz u.s.c. at $\bar p$ and the following estimate holds
$$
  \Lipusc{\Solv}{\bar p}\le{\Lipparz{F}{\bar p}
  \over\tau_{\bar p}}.
$$
\end{theorem}

\begin{proof}
By virtue of hypothesis (iv), it is possible to pick any $\tau\in (0,\tau_{\bar p})$.
According to hypothesis (iii), taken any positive $\ell>\Lipparz{F}{\bar p}$,
there exists $\delta>0$ such that
\begin{eqnarray}    \label{LipLipparzunif}
  \haus{F(p_1,x)}{F(p_2,x)}\le\ell d(p_1,p_2),\quad\forall
  p_1,\, p_2\in\ball{\bar p}{\delta},\ \forall x\in X.
\end{eqnarray}
Fix an arbitrary $p\in\ball{\bar p}{\delta}\backslash\{\bar p\}$.
If $\Solv(p)=\varnothing$, the inclusion
$$
  \Solv(p)\subseteq\ball{\Solv(\bar p)}{{\ell\over\tau}d(p,\bar p)}
$$
is trivially satisfied. Otherwise, take an arbitrary $x_p\in\Solv(p)$.
Consider the function $\varphi_F(\bar p,\cdot):X\longrightarrow [0,+\infty]$.
By virtue of hypothesis (ii), Lemma \ref{lem:Phivarphi}(i)
ensures that this function is l.s.c. on $X$. It is clearly
bounded from below. Besides, since it is $F(p,x_p)\subseteq C$,
on account of inequality $(\ref{LipLipparzunif})$ one obtains
\begin{eqnarray*}
  \varphi_F(\bar p,x_p)=\exc{F(\bar p,x_p)}{C}\le
  \exc{F(\bar p,x_p)}{F(p,x_p)}+\exc{F(p,x_p)}{C}\le
  \ell d(p,\bar p),
\end{eqnarray*}
so that $\varphi_F(\bar p,x_p)\le\inf_{x\in X}\varphi_F(\bar p,x)+\ell d(p,\bar p)$.
By applying the Ekeland variational principle, with
$$
  \lambda={\ell d(p,\bar p)\over\tau}
$$
and proceeding along the same lines as in the proof of Theorem \ref{thm:calmSolv}
with obvious adaptations, one can reach immediately all assertions in the thesis.
\end{proof}

As a comment to Theorem \ref{thm:LipuscSolv}, let us note that, since it is
$$
   \sostsl{\varphi_F(\bar p,\cdot)}(\bar x)\ge\tau_{\bar p},
$$
then condition $\tau_{\bar p}>0$ is stricter than condition (iv) in
Theorem \ref{thm:calmSolv}. Consistently, the thesis of Theorem \ref{thm:LipuscSolv}
guarantees a stronger property for $\Solv$ than that of Theorem \ref{thm:calmSolv},
in consideration of Fact 2.

\begin{remark}
As mentioned,
the analysis approach pursued for achieving the results in this section reveals
connections with the study of local error bound properties. Let us recall that,
for an extended real-valued function $\psi:X\longrightarrow\R\cup\{+\infty\}$
defined on a metric space, the error bound property is defined by the
inequality
\begin{equation}     \label{in:locerbopsi}
   \dist{x}{[\psi\le 0]}\le\gamma[\psi(x)]_+,
\end{equation}
where $[r]_+=\max\{r,0\}$, with $r\in\R\cup\{+\infty\}$. More precisely,
$\psi$ is said to admit a local error bound at $\bar x\in\psi^{-1}(0)$
if there exist $\gamma\ge 0$ and $\delta>0$ such that inequality
$(\ref{in:locerbopsi})$ holds for every $x\in\ball{\bar x}{\delta}$.
Since the values taken by $\Solv$ are reformulated as sublevel sets
of $\varphi_F$ in $(\ref{eq:levsetupinvchar})$, the excess function
$\varphi_F$ is the key element to catch the aforementioned connection,
playing the role of $\psi$ in $(\ref{in:locerbopsi})$.
While this connection leaves open the possibility of mutual benefits
for both the topics in future investigations, to the best of the author's
knowledge none of the existing error bound conditions can be
applied in the current context, because of the peculiar form taken by
$\varphi_F$. Indeed, if it is true that a well developed theory of
error bounds for function of the form
$$
  \psi(x)=\sup_{t\in T}\psi_t(x)
$$
already exists (see, for instance, \cite{KrNgTh10}),  it
requires the index set $T$ to be a fixed, compact Hausdorff space,
whereas in the definition of $\varphi_F$ the supremum must be taken
over a set not necessarily compact and depending on $x$.
\end{remark}


\section{Consequences on the value analysis in parametric optimization} \label{Sect:4}

Let $\vartheta:P\times X\longrightarrow\R$ be a given function defined
on the Cartesian product of two metric spaces $P$ and $X$, and let
$(\mathcal{SVI}_p)$ be a given parameterized class of set-valued inclusions.
In the current section, some consequences of the findings exposed in
Section \ref{Sect:3} will be presented, with reference to the
analysis of the parametric class of constrained optimization problems defined
by the aforegiven data, namely
$$
  \min\vartheta(p,x)
  \ \hbox{ with $x\in X$ subject to }\ F(p,x)\subseteq C.
  \leqno (\mathcal{P}_p)
$$
The feasible region of each problem $(\mathcal{P}_p)$
is given by the set-valued mapping $\Solv:P\rightrightarrows X$,
defined as in $(\ref{eq:Solvdef})$.
More precisely, the investigations will focus on the calmness properties
of the optimal value (alias, performance) function
$\val:P\longrightarrow\R\cup\{\pm\infty\}$, which can be associated
with $(\mathcal{P}_p)$, i.e.
$$
  \val(p)=\inf_{x\in\Solv(p)}\vartheta(p,x).
$$
A further element appearing in what follows is the solution mapping
${\rm Argmin}:P\rightrightarrows X$ associated with $(\mathcal{P}_p)$,
i.e.
$$
  \Argmin{p}=\{x\in\Solv(p):\ \vartheta(p,x)\le\val(p)\}.
$$
As one expects, in consideration of the broad spectrum of applications
promised by a similar topic, a wide literature flourished on that
subject, yielding a large amount of results, often tailored on the base
of the problem format.
One of the key reference in the value analysis, for optimization
problems with an abstract feasible region formalized by a set-valued
mapping depending on a parameter, is the so-called maximum Berge's theorem (see \cite[Theorem 17.31]{AliBor06}),
which provides a sufficient condition for the continuity of the
optimal value function in a pure topological setting.
Advances in this direction were obtained with \cite[Theorem 3B.5]{DonRoc14}.
Here a result about calmness of $\val$ is proposed, which is
specific for the problem format $(\mathcal{P}_p)$.

As in Section \ref{Sect:3}, throughout the current section it is
assumed that $\dom F=P\times X$. Besides, as a Cartesian product metric space,
$P\times X$ will be supposed to be equipped with the $\max$ distance.
By exploiting the same arguments as in the proof of
\cite[Proposition 3.2]{Uder12} (which makes only assumptions
on $\vartheta$ and $\Solv$, independently of how
the constraints defining $\Solv$ are formalized), one can
establish the following result, where Lipschitz lower semicontinuity
plays a crucial role.

\begin{proposition}[Calmness from above of $\val$]     \label{pro:ucalmval}
With reference to a parametric class of problem $(\mathcal{P}_p)$,
let $\bar p\in P$ and let $\bar x\in\Argmin{\bar p}$. Suppose that:

(i) $(X,d)$ is metrically complete;

(ii) the mapping $F(\cdot,\bar x):P\rightrightarrows Y$ is
Lipschitz u.s.c. at $\bar p$;

(iii) there exists $\delta>0$ such that, for every $p\in\ball{\bar p}{\delta}$,
each set-valued mapping $F(p,\cdot):X\rightrightarrows Y$ is l.s.c. on $X$;

(iv) it holds $\sostslx{\varphi_F}(\bar p,\bar x)>0$;

(v) function $\vartheta:P\times X\longrightarrow\R$ is calm from
above at $(\bar p,\bar x)$.

\noindent Then, function $\val:P\longrightarrow\R\cup\{\pm\infty\}$
is calm from above at $\bar p$ and
\begin{equation}   \label{in:uclmvalest}
  \ucalm{\val}{\bar p}\le\ucalm{\vartheta}{\bar p,\bar x}\cdot
  \max\left\{1,\, {\Lipusc{F(\cdot,\bar x)}{\bar p}
  \over  \sostslx{\varphi_F}(\bar p,\bar x)}\right\}.
\end{equation}
\end{proposition}

\begin{proof}
By hypothesis (v), fixed any $\kappa>\ucalm{\vartheta}{\bar p,\bar x}$,
there exists $\delta>0$ such that
\begin{equation}     \label{in:varthetacalmabove}
   \vartheta(p,x)-\vartheta(\bar p,\bar x)\le\kappa\max\{d(p,\bar p),
   d(x,\bar x)\},\quad\forall (p,x)\in\ball{\bar p}{\delta}\times
   \ball{\bar x}{\delta}.
\end{equation}
Under hypotheses (i)-(iv), Theorem \ref{thm:LiplscSolv} ensures that
$\Solv$ is Lipschitz l.s.c. at $(\bar p,\bar x)$, with modulus
satisfying inequality $(\ref{in:LiplscSolv})$. Therefore, fixed any
$\ell>\Lipusc{F(\cdot,\bar x)}{\bar p}/\sostslx{\varphi_F}(\bar p,\bar x)$,
there exists $\zeta>0$ such that
$$
  \Solv(p)\cap\ball{\bar x}{\ell d(p,\bar p)}\ne\varnothing,\quad\forall
  p\in\ball{\bar p}{\zeta}.
$$
This means that, for any $p\in\ball{\bar p}{\zeta}$, there must exist
$x_p\in\Solv(p)$ such that $d(x_p,\bar x)\le\ell d(p,\bar p)$.
Without loss of generality, one can assume $\zeta<\min\{\delta,\,
\delta/\ell\}$. Consequently, one has $x_p\in\ball{\bar x}{\delta}$,
and hence $(p,x_p)\in\ball{\bar p}{\delta}\times\ball{\bar x}{\delta}$.
This fact allows one to invoke inequality $(\ref{in:varthetacalmabove})$.
Therefore, it follows
\begin{eqnarray*}
  {\val(p)-\val(\bar p)\over d(p,\bar p)} &\le&
  {\vartheta(p,x_p)-\vartheta(\bar p,\bar x)\over d(p,\bar p)}
  \le \kappa\cdot {\max\{d(p,\bar p),d(x_p,\bar x)\}\over d(p,\bar p)}  \\
  &\le &\kappa\cdot\max\{1,\ell\},\quad\forall p\in
  \ball{\bar p}{\zeta}\backslash\{\bar p\}.
\end{eqnarray*}
The above inequality chain allows one to obtain
$$
  \limsup_{p\to\bar p} {\val(p)-\val(\bar p)\over d(p,\bar p)}
  \le\kappa\cdot\max\{1,\ell\}<+\infty.
$$
This shows that the function $\val$ is calm from above at $\bar p$, with
$\ucalm{\val}{\bar p}\le\kappa\cdot\max\{1,\ell\}$. To conclude the proof,
the inequality $(\ref{in:uclmvalest})$ can be achieved by arbitrariness of
$\kappa>\ucalm{\vartheta}{\bar p,\bar x}$ and of
$\ell>\Lipusc{F(\cdot,\bar x)}{\bar p}/\sostslx{\varphi_F}(\bar p,\bar x)$.
\end{proof}

The counterpart of the above result for the calmness from below of
$\val$ is established next by exploiting the Lipschitz upper semicontinuity
property of $\Solv$.

\begin{proposition}[Calmness from below of $\val$] \label{pro:lcalmval}
With reference to a parametric class of problem $(\mathcal{P}_p)$,
let $\bar p\in P$ and let $\bar x\in\Argmin{\bar p}$. Suppose that:

(i) $(X,d)$ is metrically complete;

(ii) the set-valued mapping $F(\bar p,\cdot):X\rightrightarrows Y$
is l.s.c. on $X$;

(iii) the mapping $F$ is locally Lipschitz near $\bar p$ with respect to $p$,
uniformly in $x\in X$;

(iv) it is $\tau_{\bar p}>0$;

(v) function $\vartheta:P\times X\longrightarrow\R$ is Lipschitz
continuous on $P\times X$.

\noindent Then, function $\val:P\longrightarrow\R\cup\{\pm\infty\}$
is calm from below at $\bar p$ and
\begin{equation}    \label{in:lcalmvalest}
  \lcalm{\val}{\bar p}\le\Lip{\vartheta}{P}{X}\cdot\max\left\{1,\,
  {\Lipparz{F}{\bar p}\over\tau_{\bar p}}\right\}.
\end{equation}
\end{proposition}

\begin{proof}
By hypothesis (v), fixed any $\kappa>\Lip{\vartheta}{P}{X}$,
one has
\begin{equation}  \label{in:Lipvarthetacalmbelow}
   |\vartheta(p_1,x_1)-\vartheta(p_2,x_2)|\le\kappa\cdot\max
   \{d(p_1,p_2),d(x_1,x_2)\},\quad\forall (p_1,x_1),\, (p_2,x_2)
   \in P\times X.
\end{equation}
Since under hypotheses (i)-(iv) one can invoke Theorem \ref{thm:LipuscSolv},
the set-valued mapping $\Solv$ turns out to be Lipschitz u.s.c. at $\bar p$,
with $\Lipusc{\Solv}{\bar p}\le\Lipparz{F}{\bar p}/\tau_{\bar p}$.
Accordingly, fixed an arbitrary $\ell>\Lipparz{F}{\bar p}/\tau_{\bar p}$,
there exists $\delta>0$ such that
$$
  \dist{x}{\Solv(\bar p)}\le\ell d(p,\bar p),\quad\forall p\in
  \ball{\bar p}{\delta}.
$$
The last inequality means that fixed $\epsilon>0$, for every $x\in\Solv(p)$
there exists $z_x\in\Solv(\bar p)$ such that
$$
  d(z_x,x)\le(\ell+\epsilon)d(p,\bar p),\quad\forall p\in\ball{\bar p}{\delta}.
$$
By using inequality $(\ref{in:Lipvarthetacalmbelow})$, one finds
\begin{eqnarray*}
  \vartheta(p,x) &\ge &\vartheta(\bar p,z_x)-\kappa\cdot\max\{d(p,\bar p),d(x,z_x)\}\\
  &\ge &\vartheta(\bar p,\bar x)-\kappa\cdot\max\{1, (\ell+\epsilon)\}d(p,\bar p),
  \quad\forall x\in\Solv(p),\quad\forall p\in\ball{\bar p}{\delta},
\end{eqnarray*}
whence it follows
\begin{eqnarray*}
  {\val(p)-\val(\bar p)\over d(p,\bar p)} &=&
  {\inf_{x\in\Solv(p)}\vartheta(p,x)-\vartheta(\bar p,\bar x)
  \over d(p,\bar p)} \\
  &\ge& -\kappa\cdot\max\{1, (\ell+\epsilon)\}>-\infty,\quad\forall
 p\in\ball{\bar p}{\delta}\backslash\{\bar p\}.
\end{eqnarray*}
By passing to the $\liminf$ as $p\to\bar p$, the last inequality
shows that function $\val$ is calm from below at $\bar p$ and
$\lcalm{\val}{\bar p}\le\kappa\cdot\max\{1, (\ell+\epsilon)\}$.
The arbitrariness of $\ell$, $\epsilon>0$ and $\kappa>\Lip{\vartheta}{P}{X}$
allows one to achieve the estimate in $(\ref{in:lcalmvalest})$, thereby
completing the proof.
\end{proof}

By combining the previous results of this section, the following condition
ensuring the calmness of $\val$ can be achieved.

\begin{theorem}[Calmness of $\val$]
With reference to a parametric class of problem $(\mathcal{P}_p)$,
let $\bar p\in P$ and let $\bar x\in\Argmin{\bar p}$. Suppose that:

(i) $(X,d)$ is metrically complete;

(ii) $F$ is locally Lipschitz near $\bar p$ with respect to $p$,
uniformly in $x\in X$;

(iii) there exists $\delta>0$ such that, for every $p\in\ball{\bar p}{\delta}$,
each $F(p,\cdot):X\rightrightarrows Y$ is l.s.c. on $X$;

(iv) it is $\min\{\sostslx{\varphi_F}(\bar p,\bar x),\tau_{\bar p}\}>0$;

(v) function $\vartheta:P\times X\longrightarrow\R$ is Lipschitz
continuous on $P\times X$.

\noindent Then, function $\val:P\longrightarrow\R\cup\{\pm\infty\}$
is calm at $\bar p$ and it holds
\begin{equation}    \label{in:calmvalest}
  \fulcalm{\val}{\bar p}\le\Lip{\vartheta}{P}{X}\cdot\max\left\{1,\,
  {\Lipparz{F}{\bar p}\over \min\{\sostslx{\varphi_F}
  (\bar p,\bar x),\tau_{\bar p}\}}\right\}.
\end{equation}
\end{theorem}

\begin{proof}
In order to prove that $\val$ is calm at $\bar p$ within the proposed
approach, one needs to check that, under the current hypotheses (i)-(v), it is
possible to apply both Proposition \ref{pro:lcalmval} and Proposition
\ref{pro:ucalmval}.
To this aim, let us start with observing that, if the set-valued mapping
$F$ is locally Lipschitz near $\bar p$ with respect to $p$,uniformly in $x\in X$,
then, in particular, the set-valued mapping $p\leadsto F(p,\bar x)$ is
Lipschitz u.s.c. at $\bar p$, with $\Lipusc{F(\cdot,\bar x)}{\bar p}
\le\Lipparz{F}{\bar p}$ (remember Fact 6). Secondly, observe that hypothesis (iii)
entails, in particular, that the set-valued mapping $x\leadsto F(\bar p,x)$
is l.s.c. on $X$. Hypothesis (iv) clearly implies that the condition
(iv) of both Proposition \ref{pro:lcalmval} and Proposition \ref{pro:ucalmval}
is fulfilled.
Finally, the Lipschitz continuity of $\vartheta$ on $P\times X$
evidently forces the calmness from above of $\vartheta$ at $(\bar p,\bar x)$,
with $\ucalm{\vartheta}{\bar p,\bar x}\le\Lip{\vartheta}{P}{X}$.
Thus, according to Proposition \ref{pro:ucalmval}, corresponding to
\begin{equation}    \label{in:lcalmvalell1}
   \ell_1>\Lip{\vartheta}{P}{X}\cdot\max\left\{1,\,
  {\Lipparz{F}{\bar p}\over \sostslx{\varphi_F}{(\bar p,\bar x)}}
  \right\},
\end{equation}
there exists $\delta_1>0$ such that
$$
  \val(p)-\val(\bar p)\le \ell_1d(p,\bar p),\quad\forall p
  \in\ball{\bar p}{\delta_1}.
$$
According to Proposition \ref{pro:lcalmval}, corresponding to
\begin{equation}    \label{in:ucalmvalell2}
  \ell_2>\Lip{\vartheta}{P}{X}\cdot\max\left\{1,\,
  {\Lipparz{F}{\bar p}\over \tau_{\bar p}}\right\}
\end{equation}
there exists $\delta_2>0$ such that
$$
  \val(p)-\val(\bar p)\ge -\ell_2d(p,\bar p),\quad\forall p
  \in\ball{\bar p}{\delta_2}.
$$
Therefore, by setting $\delta_0=\min\{\delta_1,\delta_2\}$
and $\ell_0=\max\{\ell_1,\ell_2\}$, one can assert that
$$
  |\val(p)-\val(\bar p)|\le\ell_0d(p,\bar p),\quad\forall p
  \in\ball{\bar p}{\delta_0}.
$$
This shows that $\val$ is calm at $\bar p$ and that $\fulcalm{\val}{\bar p}
\le\ell_0$. Since $\ell_1$ and $\ell_2$ can be chosen to be arbitrarily
closed to the right term in inequalities $(\ref{in:lcalmvalell1})$
and $(\ref{in:ucalmvalell2})$, respectively, one can conclude that inequality
$(\ref{in:calmvalest})$ holds.
\end{proof}

\vskip1cm


\section{Some special conditions in Banach spaces} \label{Sect:5}

Even though the differential conditions appearing in Theorem \ref{thm:LiplscSolv},
Theorem \ref{thm:calmSolv} and Theorem \ref{thm:LipuscSolv} have a transparent
meaning in metric spaces, they need to be further worked in view of
effective employments in more structured settings.
The aim of the present section is therefore to provide
useful (that is, from below) estimates for the three constants
\begin{equation}   \label{eq:dicondconsts}
 \sostslx{\varphi_F}(\bar p,\bar x),\qquad
 \sostsl{\varphi_F(\bar p,\cdot)}(\bar x),\qquad \hbox{ and }\qquad
 \tau_{\bar p},
\end{equation}
which are directly based on the problem data ($F$ and $C$).
A similar question already arose in the study of quantitative stability
properties of the solution set to traditional generalized equations
and has been solved with the aid of (sometimes, ad hoc) involved constructions
of nonsmooth analysis, such as graphical derivatives, prederivatives,
coderivatives, estimators (see, for instance,
\cite{BonSha00,BorZhu05,DonRoc14,Ioff81,Ioff18,KlaKum02,Mord06,Mord18}).
Because of the expression of $\varphi_F$, existing results suitable for
traditional generalized equations seem not be immediately exploitable
within the proposed approach to the problem at the issue.
What follows must be regarded as a first attempt to address the question,
starting with basic tools of convex analysis. It is reasonable to
believe that the employment of more involved constructions of nonsmooth
analysis, already available, might enlarge the class of set-valued
inclusions, for which useful estimates can be established, and
afford deeper insights into this topic.

\begin{definition}[Concave mapping]    \label{def:concavesvmap}
A set-valued mapping $\Phi:\X\rightrightarrows\Y$ between Banach spaces
is said to be {\it concave} on $\X$ if it holds
\begin{equation}     \label{def:concavpro}
  \Phi(tx_1+(1-t)x_2)\subseteq t\Phi(x_1)+(1-t)\Phi(x_2),\quad
  \forall x_1,\, x_2\in\X,\ \forall t\in [0,1].
\end{equation}
\end{definition}

A generalization of the notion of concavity for set-valued mappings
introduced in Definition \ref{def:concavesvmap} has been already
considered, in connection with set-valued inclusions, in \cite{Cast99}.

\begin{example}[Fan]     \label{ex:fan}
After \cite{Ioff81}, a set-valued mapping $\Phi:\X\rightrightarrows\Y$
between Banach spaces is said to be a fan if all the following
conditions are fulfilled:

(i) $\nullv\in\ \Phi(\nullv)$;

(ii) $\Phi(t x)=t \Phi(x)$, $\forall x\in\X$ and $\forall t>0$;

(iii) $\Phi$ takes convex values;

(iv) $\Phi(x_1+x_2)\subseteq \Phi(x_1)+\Phi(x_2)$, $\forall x_1,\, x_2\in\X$.

\noindent By virtue of conditions (ii) and (iv), it is clear that any fan is
a positively homogeneous concave set-valued mapping. As a particular example of fan,
one can consider set-valued mappings which are generated by families
of linear bounded operators. More precisely, let $\mathcal{G}\subseteq\Lin
(\X,\Y)$ be a convex set weakly closed with respect to the weak topology
on $\Lin(\X,\Y)$. Define
$$
  \Phi_\mathcal{G}(x)=\{y\in\Y:\ y=\Lambda x,\, \Lambda\in\mathcal{G}\}.
$$
The set-valued mapping $\Phi_\mathcal{G}:\X\rightrightarrows\Y$ is known to
be a particular example of fan. Note, however, that there are fans, which can not
be generated by any family of linear bounded operators. The set-valued mapping
$\Phi$ considered in Example \ref{ex:PHvsLipusc}(ii) provides an instance of
such a circumstance.
\end{example}

For other examples of concave set-valued mappings see \cite{Uder20}.

The next lemma shows that the assumption of concavity on a set-valued
mapping $\Phi$ entails convenient properties of the related excess function
$\varphi_\Phi$, which allow one to carry out the approach here proposed
by tools of convex analysis.

\begin{lemma}    \label{lem:concaveconvexexc}
Let $\Phi:\X\rightrightarrows\Y$ be a set-valued mapping between
Banach spaces and let $C\subseteq\Y$ a closed, convex cone.
Then,

(i) if $\Phi$ is positively homogeneous on $\X$, so is $\varphi_\Phi$;

(ii) if $\Phi$ is concave on $\X$, $\varphi_\Phi$ is convex;

(iii) if $\Phi$ is superlinear (positively homogeneous and concave)
on $\X$, $\varphi_\Phi$ is sublinear.
\end{lemma}

\begin{proof}
First of all, recall that the function $y\mapsto\dist{y}{C}$, as a distance
function from a convex cone, is sublinear on $\Y$.

(i) One has
\begin{eqnarray*}
  \varphi_\Phi(tx) &=& \sup_{y\in\Phi(tx)}\dist{y}{C}=\sup_{y\in\Phi(x)}
    \dist{ty}{C}=\sup_{y\in\Phi(x)}t\dist{y}{C}  \\
    &=&t\varphi_\Phi(x), \quad\forall t>0,\ \forall x\in\dom\Phi.
\end{eqnarray*}

(ii) Fix $x_1,\, x_2\in\dom\Phi$ and $t\in [0,1]$. By using the concavity
of $\Phi$ and the sublinearity of function $y\mapsto\dist{y}{C}$, one
obtains
\begin{eqnarray*}
   \varphi_\Phi(tx_1+(1-t)x_2) &=& \sup_{y\in\Phi(tx_1+(1-t)x_2)}\dist{y}{C}\le
    \sup_{y\in t\Phi(x_1)+(1-t)\Phi(x_2)}\dist{y}{C} \\
    &=& \sup_{y_1\in\Phi(x_1),\, y_2\in\Phi(x_2)}\dist{ty_1+(1-t)y_2}{C}  \\
    &\le& \sup_{y_1\in\Phi(x_1),\, y_2\in\Phi(x_2)} [t\dist{y_1}{C}+(1-t)\dist{y_2}{C}] \\
    &=& t\sup_{y_1\in\Phi(x_1)}\dist{y_1}{C}+(1-t)\sup_{y_2\in\Phi(x_2)}\dist{y_2}{C} \\
    &=& t\varphi_\Phi(x_1)+(1-t)\varphi_\Phi(x_2).
 \end{eqnarray*}

(iii) This assertion is a straightforward consequence of the above assertions
(i) and (ii).
\end{proof}


A quantitative behaviour, which can be regarded as a counterpart
of the metric decrease property for set-valued mappings taking values
in a partially ordered Banach space, is captured by the next definition.

\begin{definition}[Metric $C$-increase]     \label{def:MCincmap}
Given a closed, convex cone $C\subseteq\Y$, a set-valued mapping
$\Phi:\X\rightrightarrows\Y$ between Banach spaces is said to be

\begin{itemize}

\item[(i)] {\it metrically $C$-increasing}
on $\X$ if there exists a constant $\alpha>1$ such that
\begin{equation}   \label{in:metincdefglo}
   \forall x\in\X,\ \forall r>0,\
   \exists u\in\ball{x}{r}:\ \ball{\Phi(u)}{\alpha r}
     \subseteq\ball{\Phi(x)+C}{r};
\end{equation}
the quantity
$$
   \incr{\Phi}=\sup\{\alpha>1:\ \hbox{ inclusion  $(\ref{in:metincdefglo})$
   holds}\}
$$
is called {\it exact bound of metric $C$-increase} of $\Phi$ on $\X$.

\item[(ii)] {\it metrically $C$-increasing}
around $\bar x\in\dom\Phi$ if there exist $\delta>0$ and $\alpha>1$ such that
\begin{equation}   \label{in:metincdefloc}
   \forall x\in\ball{\bar x}{\delta},\ \forall r\in (0,\delta),\
   \exists u\in\ball{x}{r}:\ \ball{\Phi(u)}{\alpha r}
     \subseteq\ball{\Phi(x)+C}{r};
\end{equation}
the quantity
$$
    \incr{\Phi}(\bar x)=\sup\{\alpha>1:\ \exists\delta>0\ \hbox{such that
    inclusion $(\ref{in:metincdefloc})$ holds}\}
$$
is called {\it exact bound of metric $C$-increase} of $\Phi$ around $\bar x$.
\end{itemize}
Let $\Phi:P\times\X\rightrightarrows\Y$ be a set-valued mapping, defined
on the product of a metric space $P$ with a Banach space $\X$, and
taking values in a Banach space $\Y$, and let $(\bar p,\bar x)\in P\times\X$.
$\Phi$ is said to be
\begin{itemize}

\item[(iii)] {\it metrically $C$-increasing with respect to $x$} around
$(\bar p,\bar x)$, {\it uniformly in $p$}, if there exist $\delta>0$ and
$\alpha>1$ such that
\begin{equation}   \label{in:unifmetincdefglo}
   \forall (p,x)\in\ball{\bar p}{\delta}\times\ball{\bar x}{\delta},\
   \forall r\in (0,\delta),\ \exists u\in\ball{x}{r}:\ \ball{\Phi(p,u)}{\alpha r}
     \subseteq\ball{\Phi(p,x)+C}{r};
\end{equation}
the quantity
$$
   \incr{\Phi}_x(\bar p,\bar x)=\sup\{\alpha>1:\ \hbox{ inclusion
   $(\ref{in:unifmetincdefglo})$  holds}\}
$$
is called {\it exact uniform bound of metric $C$-increase} of $\Phi$
near $(\bar p,\bar x)$.
\end{itemize}
\end{definition}

The above properties have been already used in connection with the study
of error bounds for set-valued inclusions in \cite{Uder19}, where
several examples of the occurrence of the metric $C$-increase property
in global and local form can be found.

\begin{remark}    \label{rem:excsupfprop}
In the proof of the next proposition, the following facts concerning properties
of the excess and the support function will be employed.
Let $S\subseteq\Y$ be a nonempty subset, let $C\subseteq\Y$ be a closed, convex
cone, and let $r>0$. Then, the following equalities hold:

(i) $\exc{S+C}{C}=\exc{S}{C}$ (see \cite[Remark 2.1]{Uder19});

(ii) if $\exc{S}{C}>0$, then $\exc{\ball{S}{r}}{C}=\exc{S}{C}+r$
(see \cite[Lemma 2.2]{Uder19}).

\noindent Let $S\subseteq\X^*$ be a closed, convex set, and let $\supf{\cdot}{S}
:\X\longrightarrow\R\cup\{\pm\infty\}$  denote its support function,
i.e. $\supf{x}{S}=\sup_{x^*\in S}\langle x^*,x\rangle$, where
$\langle\cdot,\cdot\rangle$ denotes the duality pairing the space
$\X$ and its dual $\X^*$.
Then, $\nullv^*\in S$ iff $[\supf{\cdot}{S}\ge 0]=\X$ and, more
precisely, the following distance estimate holds
\begin{equation}     \label{in:distestsupf}
  \dist{\nullv^*}{S}\ge-\inf_{u\in\B}\supf{u}{S}
\end{equation}
(see, for instance, \cite[Remark 2.1]{Uder20}).
\end{remark}

The next proposition explains the role of the property introduced in
Definition \ref{def:MCincmap} within the present approach.
Roughly speaking, it provides a method for measuring the violation of
the set-valued inclusion $\Phi(x)\subseteq C$ near a solution
$\bar x\in\Phi^{+1}(C)$. This is done by tools of convex analysis both
in the primal space (via the directional derivative of $\varphi_\Phi$)
and in the dual space (via the subdifferential of $\varphi_\Phi$).
Recall that, given a convex function $\psi:\X\longrightarrow\R\cup\{\pm\infty\}$
and a point $x\in\psi^{-1}(\R)$ the value
$$
  \psi'(x;v)=\lim_{t\to 0^+}{\psi(x+tv)-\psi(x)\over t},
$$
is called directional derivative of $\psi$ at $x$, in the direction
$v\in\X$.

\begin{proposition}   \label{pro:subdifslinc}
Let $\Phi:\X\rightrightarrows\Y$ be a set-valued mapping between Banach spaces,
let $C\subseteq\Y$ a closed, convex cone and let $\bar x\in\X$ such that
$\Phi(\bar x)\subseteq C$. Suppose that:

(i) $\Phi$ is l.s.c., concave and bounded-valued away from $C$ on $\X$;

(ii) $\Phi$ is metrically $C$-increasing around $\bar x$, with exact bound
$\incr{\Phi}(\bar x)$.

\noindent Then, there exists $\eta>0$ such that
\begin{equation}   \label{in:indirderloc}
    \inf_{u\in\B}\varphi_\Phi'(x;u)\le 1-\incr{\Phi}(\bar x),\quad
    \forall x\in\ball{\bar x}{\eta}\cap[\varphi_\Phi>0],
\end{equation}
and hence
\begin{equation} \label{in:subdifslincloc}
 \dist{\nullv^*}{\partial\varphi_\Phi(x)}\ge\incr{\Phi}(\bar x)-1>0,\quad
 \forall x\in\ball{\bar x}{\eta}\cap[\varphi_\Phi>0].
\end{equation}

\noindent If hypothesis (ii) is replaced with

(ii)' $\Phi$ is metrically $C$-increasing on $\X$, with exact bound
$\incr{\Phi}$,

\noindent then, one has
\begin{equation}
    \inf_{u\in\B}\varphi_\Phi'(x;u)\le 1-\incr{\Phi},\quad
    \forall x\in[\varphi_\Phi>0],
\end{equation}
and hence
\begin{equation}
 \dist{\nullv^*}{\partial\varphi_\Phi(x)}\ge\incr{\Phi}-1>0,\quad
 \forall x\in[\varphi_\Phi>0].
\end{equation}
\end{proposition}

\begin{proof}
First of all observe that, on account of Lemma \ref{lem:Phivarphi}(i), Lemma
\ref{lem:concaveconvexexc}(ii), and Remark \ref{rem:varphirealval}, by hypothesis (i)
the function $\varphi_\Phi:\X\longrightarrow\R$ is l.s.c., convex and
$\varphi_\Phi^{-1}(\R)=\X$.

Fix an arbitrary $\alpha\in(1,\incr{\Phi}(\bar x))$ and let $\delta>0$ be
as in Definition \ref{def:MCincmap}(ii). Take an arbitrary $x_0\in\ball{\bar x}{\delta}
\cap[\varphi_\Phi>0]$. Since it is $\varphi_\Phi(x_0)>\varphi_\Phi(\bar x)=0$, the point
 $x_0$ can not be a (global) minimizer of $\varphi_\Phi$.
Consequently, as the inclusion $\nullv^*\in\partial\varphi_\Phi(x_0)$
characterizes the minimality for a convex function, it must be
$\nullv^*\not\in\partial\varphi_\Phi(x_0)$.
Since function $\varphi_\Phi$ is l.s.c. on $\X$, the superlevel set $[\varphi_\Phi>0]$
turns out to be open, so there exists $\delta_0>0$ such that $\ball{x_0}{\delta_0}
\subseteq [\varphi_\Phi>0]$, what means that
$$
  \Phi(x)\not\subseteq C,\quad\forall x\in\ball{x_0}{\delta_0}.
$$
Now, according to hypothesis (ii), for every $t\in (0,\delta_0)$ there exists
$u_t\in\B$ such that
$$
  \ball{\Phi(x_0+tu_t)}{\alpha t}\subseteq\ball{\Phi(x_0)+C}{t},
$$
while $\Phi(x_0+tu_t)\not\subseteq C$. By taking into account the
equalities in Remark \ref{rem:excsupfprop}(i) and (ii), one obtains
\begin{eqnarray}     \label{in:phiatx0}
   \exc{\ball{\Phi(x_0+tu_t)}{\alpha t}}{C}  &\le&
   \exc{\ball{\Phi(x_0)+C}{t}}{C} =
   \exc{\Phi(x_0)+C}{C}+t   \\  \nonumber
   &=&\exc{\Phi(x_0)}{C}+t = \varphi_\Phi(x_0)+t.
\end{eqnarray}
On the other hand, one has
\begin{equation}    \label{eq:phiaround}
   \exc{\ball{\Phi(x_0+tu_t)}{\alpha t}}{C} =\exc{\Phi(x_0+tu_t)}{C}
   +\alpha t = \varphi_\Phi(x_0+tu_t)+\alpha t.
\end{equation}
From equality $(\ref{eq:phiaround})$ and the relations in $(\ref{in:phiatx0})$,
one deduces
$$
  \inf_{u\in\B}{\varphi_\Phi(x_0+tu)-\varphi_\Phi(x_0)\over t}\le
  {\varphi_\Phi(x_0+tu_t)-\varphi(x_0)\over t}\le 1-\alpha,
  \quad\forall t\in (0,\delta_0).
$$
Therefore, it results in
$$
   \inf_{u\in\B}\varphi_\Phi'(x_0;u)=\inf_{u\in\B}\inf_{t\in (0,\delta_0)}
   {\varphi_\Phi(x_0+tu)-\varphi_\Phi(x_0)\over t}=
   \inf_{t\in (0,\delta_0)}\inf_{u\in\B}{\varphi_\Phi(x_0+tu)-\varphi_\Phi(x_0)\over t}
   \le 1-\alpha.
$$
Thus, by arbitrariness of $\alpha\in(1,\incr{\Phi})$ and $x_0\in\ball{\bar x}{\delta}
\cap[\varphi_\Phi>0]$, it suffices to set $\eta=\delta$ to obtain inequality
$(\ref{in:indirderloc})$. The estimate in $(\ref{in:subdifslincloc})$ can be
established by exploiting the inequality $(\ref{in:distestsupf})$ in
Remark \ref{rem:excsupfprop} and by recalling the Moreau-Rockafellar
representation formula for the directional derivative of a l.s.c. convex
function
$$
  \varphi_\Phi'(x;u)=\supf{u}{\partial\varphi_\Phi(x)},\quad\forall u\in\X,
$$
which is valid for every $x\in\inte(\dom\varphi_\Phi)=\X$ (see, for instance,
\cite[Theorem 4.2.7]{BorZhu05}). Here $\inte S$ denotes the (topological)
interior of a set $S$. Take into account that for a l.s.c. convex function
the interior of its domain coincides with the core of the domain
(see, for instance, \cite[Theorem 4.1.8]{BorZhu05}).

The second part of the thesis, upon hypothesis (ii)', can be proved
in a similar manner, with plane adaptations.
\end{proof}

Let us come back now to the context of parameterized set-valued inclusions
$(\mathcal{SVI}_p)$. It will be assumed henceforth that for any $p\in P$
near $\bar p$, the set-valued mapping $x\leadsto F(p,x)$ is concave.
Notice that, upon such an assumption,
an appreciable consequence of equality $(\ref{eq:levsetupinvchar})$
is that $\Solv:P\rightrightarrows\X$ is convex-valued, in the light
of Lemma \ref{lem:concaveconvexexc}(ii). Moreover, some estimates
of the constants in $(\ref{eq:dicondconsts})$ can be obtained
via exact bounds of metric $C$-increase, as stated below.

\begin{theorem}    \label{thm:constestinc}
Let $F:P\times\X\rightrightarrows\Y$ be a set-valued mapping, defined
on the product of a metric space $P$ with a Banach space $\X$, and
taking values in a Banach space $\Y$. Let $C\subseteq\Y$ be a closed,
convex cone, let $\bar p\in P$ and let $\bar x\in\Solv(\bar p)$.
Suppose that each set-valued mapping $F(p,\cdot):\X\rightrightarrows
\Y$ is l.s.c., concave and bounded-valued away from $C$ on $\X$,
for every $p\in\ball{\bar p}{\delta}$, for some $\delta>0$.

(i) If $F$ is metrically $C$-increasing with respect to $x$ around
$(\bar p,\bar x)$, uniformly in $p$ with exact uniform bound $\incr{F}_x
(\bar p,\bar x)$, then it holds
$$
 \sostslx{\varphi_F}(\bar p,\bar x)\ge\incr{F}_x(\bar p,\bar x)-1>0.
$$

(ii) If $F(\bar p,\cdot)$ is metrically $C$-increasing around $\bar x$,
with exact bound $\incr{F(\bar p,\cdot)}(\bar x)$, then
$$
 \sostsl{\varphi_F(\bar p,\cdot)}(\bar x)\ge
 \incr{F(\bar p,\cdot)}(\bar x)-1>0.
$$

(iii) If $F(\bar p,\cdot)$ is metrically $C$-increasing on $\X$,
with exact bound $\incr{F(\bar p,\cdot)}$, then
$$
  \tau_{\bar p}\ge \incr{F(\bar p,\cdot)}-1>0.
$$
\end{theorem}

\begin{proof}
(i) Under the current assumptions, it is possible to apply Proposition
\ref{pro:subdifslinc} to each function $F(p,\cdot):\X\rightrightarrows\Y$,
with $p\in\ball{\bar p}{\delta}$. As a consequence, there exists $\eta>0$
such that
$$
  \dist{\nullv^*}{\partial\varphi_F(p,\cdot)(x)}\ge\incr{F}_x
  (\bar p,\bar x)-1,\quad\forall x\in\ball{\bar x}{\eta}\cap
  [\varphi_F(p,\cdot)>0].
$$
Without loss of generality, one can assume $\eta<\delta$. It is to be
noticed that the value of $\eta$ is the same for each $p\in\ball{\bar p}{\delta}$,
by virtue of the unifom version of the metric $C$-increase property
postulated in hypothesis (i). Thus, in the light of Example \ref{ex:stsl}(ii),
it results in
$$
  \stsl{\varphi_F(p,\cdot)}(\bar x)=\dist{\nullv^*}{\partial
  \varphi_F(p,\cdot)(x)}\ge\incr{F}_x(\bar p,\bar x)-1,
  \quad\forall (p,x)\in\ball{\bar p}{\eta}\times\ball{\bar x}{\eta}:
  \ \varphi_F(p,x)>0.
$$
By recalling the definition in $(\ref{def:sostslx})$, one immediately obtains
\begin{eqnarray*}
  \sostslx{\varphi_F}(\bar p,\bar x) &=& \lim_{\epsilon\to 0^+}\inf
     \{\stsl{\varphi_F(p,\cdot)}(x): (p,x)\in \ball{\bar p}{\epsilon}\times
    \ball{\bar x}{\epsilon},\\
    & & \hskip4cm \varphi_F(\bar p,\bar x)<
     \varphi_F(p,x)<\varphi_F(\bar p,\bar x)+\epsilon\} \\
    &\ge & \inf\{\dist{\nullv^*}{\partial\varphi_F(p,\cdot)(x)}:\
    (p,x)\in\ball{\bar p}{\eta}\times\ball{\bar x}{\eta},\
  \ \varphi_F(p,x)>0\}  \\
  &\ge& \incr{F}_x(\bar p,\bar x)-1>0.
\end{eqnarray*}
 In the case of assertions (ii) and (iii), it suffices to apply Proposition
\ref{pro:subdifslinc} with $\Phi=F(\bar p,\cdot)$  and to recall that
$$
  \stsl{\varphi_F(\bar p,\cdot)}(x)=\dist{\nullv^*}{\partial\varphi_F
  (\bar p,\cdot)(x)}.
$$
This completes the proof.
\end{proof}

As an application of Theorem \ref{thm:constestinc}, concretely computable
estimates for some of the constants in $(\ref{eq:dicondconsts})$ are
provided in the next example, in the special case of fans generated by
linear operators between finite-dimensional Euclidean spaces.

\begin{example}
Let $\X=\R^n$ and $\Y=\R^m$ be equipped with their usual Euclidean space
structure and let $(P,d)$ be a metric space. Let $\mathcal{G}:P\rightrightarrows
\Lin(\R^n,\R^m)$ be a set-valued mapping with convex and compact values, and
with $\dom\mathcal{G}=P$. According to Example \ref{ex:fan}, $\mathcal{G}$
defines a mapping $H_\mathcal{G}:P\times\R^n\rightrightarrows\R^m$ as follows
$$
  H_\mathcal{G}(p,x)=\{\Lambda x:\ \Lambda\in\mathcal{G}(p)\}.
$$
Observe that for any $p\in P$ the set-valued mapping
$x\leadsto H_\mathcal{G}(p,x)$ is a fan, so it is concave on $\R^n$, and it is also
bounded-valued on $\R^n$, as $\mathcal{G}(p)$ is a compact set.
Moreover, it is l.s.c. at every point $\bar x\in\R^n$. Indeed, for every
open set $O\subseteq\R^m$ such that $H_\mathcal{G}(p,\bar x)\cap O\ne\varnothing$,
there must exist $\Lambda_O\in\mathcal{G}(p)$ such that $\Lambda_O\bar x\in O$.
Therefore, it suffices to consider $\Lambda_O^{-1}(O)$ to get a neighbourhood
$U$ of $\bar x$ such that $H_\mathcal{G}(p,x)\cap O\ne\varnothing$, for every
$x\in U$.
Fix $\bar p\in P$ and, taken any $\Lambda\in\mathcal{G}(\bar p)$, define
$$
   \cov\Lambda=\inf_{\|y\|=1}\|\Lambda^\top y\|,
$$
where $\Lambda^\top$ stands for the adjoint operator to $\Lambda$ (thereby
represented by the transpose to the matrix representing $\Lambda$).
It is well known in variational analysis that
$$
  \cov\Lambda=\sup\{\eta>0:\ \Lambda\B\supseteq\eta\B\},
$$
where $\B$ denotes the closed unit ball centered at $\nullv$
(see, for instance, \cite[Corollary 1.58]{Mord06}).
Let $C\subseteq\R^m$ be a closed, convex and pointed cone, such that
$\{\nullv\}\ne C\ne\R^m$ and $\inte C\ne\varnothing$.
Let us show that, if
$$
  \bar\eta=\inf_{\Lambda\in\mathcal{G}(\bar p)}\cov\Lambda>0
$$
and
\begin{equation}    \label{ne:intecond}
    \inte\left(\bigcap_{\Lambda\in\mathcal{G}(\bar p)}
    \Lambda^{-1}(C)\right)\ne\varnothing,
\end{equation}
then the set-valued mapping $H_\mathcal{G}(\bar p,\cdot)$
is metrically $C$-increasing on $\X$, with exact bound
$\incr{H_\mathcal{G}(\bar p,\cdot)}\ge\bar\eta+1$.
According to condition $(\ref{ne:intecond})$, there exist $u\in\R^n$ and
$\epsilon\in (0,1)$ such that
$$
  u+\epsilon\B\subseteq\bigcap_{\Lambda\in\mathcal{G}(\bar p)}
    \Lambda^{-1}(C).
$$
Actually, it is possible to assume $u\ne\nullv$. Indeed, if it is
$u=\nullv$ so that
$$
    \epsilon\B\subseteq\bigcap_{\Lambda\in\mathcal{G}(\bar p)}
    \Lambda^{-1}(C),
$$
for some $v\in\epsilon\B\backslash\{\nullv\}$, it is true that
$$
  \Lambda v\in C,\qquad \Lambda(-v)=-\Lambda v\in C,\qquad
  \forall \Lambda \in\mathcal{G}(\bar p).
$$
Since $C$ is a pointed cone, the above two inclusions imply that
$\Lambda v=\nullv$, for every $\Lambda \in\mathcal{G}(\bar p)$,
whence
$$
  \Lambda(v+\epsilon\B)=\Lambda(\epsilon\B)\subseteq C,
  \quad\forall \Lambda \in\mathcal{G}(\bar p).
$$
This inclusion  means that
$$
   v+\epsilon\B\subseteq\bigcap_{\Lambda\in\mathcal{G}(\bar p)}
    \Lambda^{-1}(C).
$$
So one can take $u=v\ne\nullv$. Furthermore, since the set
$\cap_{\Lambda\in\mathcal{G}(\bar p)}\Lambda^{-1}(C)$ is a cone
as an intersection of cones, it is also possible to assume that
$u\in\B\backslash\{\nullv\}$.

Take an arbitrary $\eta\in (0,\bar\eta)$. Since $\eta<
\inf_{\Lambda\in\mathcal{G}(\bar p)}\cov\Lambda$, so $\eta<
\cov\Lambda$ for every $\Lambda \in\mathcal{G}(\bar p)$, it
holds
$$
  \Lambda(\epsilon\B)\supseteq\epsilon\eta\B,\quad\forall
  \Lambda \in\mathcal{G}(\bar p).
$$
Consequently, it results in
$$
  \Lambda u+\epsilon\eta\B\subseteq\Lambda(u+\epsilon\B)
  \subseteq C,\quad\forall\Lambda \in\mathcal{G}(\bar p).
$$
By definition of $H_\mathcal{G}$, one has
$$
  H_\mathcal{G}(\bar p,u)+\epsilon\eta\B\subseteq C,
$$
wherefrom, as the set-valued mapping $H_\mathcal{G}(\bar p,\cdot)$
is positively homogeneous, it follows
$$
   H_\mathcal{G}(\bar p,\epsilon^{-1}u)+\eta\B\subseteq
   \epsilon^{-1}C=C.
$$
Now, take arbitrary $x\in\R^n$ and $r>0$. By the last inclusion,
setting $z=x+r\epsilon^{-1}u$, one has that $z\in\ball{x}{r}$ and
\begin{eqnarray*}
  \ball{H_\mathcal{G}(\bar p,z)}{(\eta+1)r} &=&
   H_\mathcal{G}(\bar p,z)+(\eta+1)r\B\subseteq H_\mathcal{G}(\bar p,x)
   +rH_\mathcal{G}(\bar p,\epsilon^{-1}u)+\eta r\B+r\B \\
   &=& H_\mathcal{G}(\bar p,x)+r(H_\mathcal{G}(\bar p,\epsilon^{-1}u)+\eta\B)+r\B \\
   &\subseteq & H_\mathcal{G}(\bar p,x)+rC+r\B=H_\mathcal{G}(\bar p,x)+C+r\B \\
   &=& \ball{H_\mathcal{G}(\bar p,x)+C}{r}.
\end{eqnarray*}
The above inclusion chain shows that $H_\mathcal{G}(\bar p,\cdot)$
is metrically $C$-increasing on $\X$, with exact bound
$\incr{H_\mathcal{G}(\bar p,\cdot)}\ge\eta+1$. By arbitrariness of
$\eta\in (0,\bar\eta)$, it results in
$$
  \incr{H_\mathcal{G}(\bar p,\cdot)}\ge\left(\inf_{\Lambda\in\mathcal{G}(\bar p)}
  \cov\Lambda\right)+1.
$$
Thus, by applying Theorem \ref{thm:constestinc}(iii), one can achieve the
estimates
$$
  \sostsl{\varphi_{H_\mathcal{G}}(\bar p,\cdot)}(\bar x)\ge\tau_{\bar p}\ge
  \inf_{\Lambda\in\mathcal{G}(\bar p)} \cov\Lambda=\inf_{\Lambda\in\mathcal{G}(\bar p)}
  \inf_{\|y\|=1}\|\Lambda^\top y\|.
$$
\end{example}


\section{Conclusions}  \label{Sect:6}

This paper contains a study on several quantitative semicontinuity properties
of the solution mapping to parameterized set-valued inclusions, which are
problems of interest in robust optimization and related fields, where set-valuedness
enters as a key feature. The findings of the exposed investigations show that,
by following a variational approach already adopted for similar properties,
even thought in the context of different problems (traditional generalized equations), it is
possible to perform an effective analysis in metric spaces. As a result of
this analysis, mainly focused on sufficient conditions for the occurrence of
the aforementioned properties, the problem of estimating certain slope constants emerges,
in the perspective of deriving verifiable conditions in more structured
settings. A first solution to this question is obtained by tools of convex
analysis. Moreover, the achievements reported seem to afford suggestions
for a more general solution, able to embed broader class of set-valued inclusions,
which consist in adapting already existent tools of nonsmooth analysis
(specific nonconvex subdifferential and related coderivatives, with a
suitable calculus).
Further development directions of the present study relate to the analysis
of other Lipschitz-type properties for the solution mapping to parameterized set-valued
inclusions, such as Aubin property and isolated calmness, following the approach
proposed in Section \ref{Sect:3}.

\vskip.5cm

\noindent{\bf Acknowledgements} The author would like to express his gratitude
to two anonymous referees for relevant remarks and useful suggestions, which
helped him to improve the quality of the paper.


\vskip1cm

\bibliographystyle{amsplain}

\begin{thebibliography}{99}

\bibitem{AliBor06} {\sc Aliprantis, C.D.} and  {\sc Border, K.C.},
\textit{Infinite dimensional analysis. A hitchhiker's guide}.
Springer, Berlin, 2006.

\bibitem{AzeCor04} {\sc Az\'e, D.} and {\sc Corvellec, J.-N.},
\textit{Characterizations of error bounds for lower semicontinuous
functions on metric spaces}, ESAIM Control Optim. Calc. Var.
\textbf{10} (2004), no. 3, 409--425.

\bibitem{BenNem98} {\sc Ben-Tal, A.} and {\sc Nemirovski,  A. },
{\it Robust convex optimization}, Math. Oper. Res. \textbf{23}(4)
(1998),  769--805.

\bibitem{BonSha00}  {\sc Bonnans, J.F.} and {\sc Shapiro, A.},
\textit{Perturbation analysis of optimization problems},
Springer-Verlag, New York, 2000.

\bibitem{BorZhu05} {\sc Borwein, J.M} and {\sc Zhu, Q.J.},
\textit{Techniques of variational analysis}, Springer-Verlag,
New York, 2005.

\bibitem{Cast99} {\sc Castellani, M.}, \textit {Error bounds for
set-valued maps}, in Generalized convexity and optimization for economic
and financial decisions, 121--135, Pitagora, Bologna, 1999.

\bibitem{CiFaKr19} {\sc Cibulka, R.}, {\sc Fabian, M.} and
{\sc Kruger, A.Y.}, \textit{On semiregularity of mappings},
J. Math. Anal. Appl. \textbf{473} (2019), no. 2, 811--836.

\bibitem{DeMaTo80} {\sc De Giorgi, E.}, {\sc Marino, A.}, and
{\sc Tosques, M.}, \textit{Problems of evolution in metric spaces
and maximal decreasing curves},  Atti Accad. Naz. Lincei Rend.
Cl. Sci. Fis. Mat. Natur. (8)  \textbf{68} (1980), no. 3, 180--187
[in Italian].

\bibitem{DonRoc14} {\sc Dontchev A.L.} and {\sc Rockafellar R.T.},
\textit{Implicit functions and solution mappings. A view from
variational analysis}, Second edition. Springer, New York, 2014.

\bibitem{FaHeKrOu10} {\sc Fabian, M.J.}, {\sc Henrion, R.}, {\sc Kruger,
A.Y.}, and {\sc Outrata, J.}, \textit{Error bounds: necessary
and sufficient conditions}, Set-Valued Var. Anal. \textbf{18} (2010),
no. 2, 121--149.

\bibitem{Hoff52} {\sc Hoffman, A.J.}, \textit{On approximate solutions
of systems of linear inequalities}, J. Research Nat. Bur. Standards
\textbf{49} (1952), 263--265.

\bibitem{Ioff81} {\sc Ioffe, A.D.}, \textit{Nonsmooth analysis:
differential calculus of nondifferentiable mappings}, Trans.
Amer. Math. Soc. \textbf{266} (1981), no. 1, 1--56.

\bibitem{Ioff00} {\sc Ioffe, A.D.}, \textit{Metric regularity and
subdifferential calculus}, Uspekhi Mat. Nauk \textbf{55} (2000), no.
3(333), 103--162.

\bibitem{Ioff18} {\sc Ioffe, A.D.}, \textit{Variational analysis of
regular mappings. Theory and applications}, Springer, Cham, 2017.

\bibitem{KhTaZa15} {\sc Khan, A.A.}, {\sc Tammer, K.} and {\sc Z\u alinescu,
C.}, \textit{Set-valued optimization. An introduction with applications},
Springer, Heidelberg, 2015.

\bibitem{KlaKum02} {\sc Klatte, D.} and {\sc Kummer, B.}, \textit{Nonsmooth
equations in optimization. Regularity, calculus, methods and applications},
Kluwer Academic Publishers, Dordrecht, 2002.

\bibitem{Krug09} {\sc Kruger, A.Y.}, \textit{About stationarity and regularity
in variational analysis}, Taiwanese J. Math. \textbf{13} (2009), no. 6A,
1737--1785.

\bibitem{KrNgTh10} {\sc Kruger, A.Y.}, {\sc Ngai, H.V.}, and {\sc Th\'era, M.},
\textit{Stability of error bounds for convex constraint systems in Banach
spaces}, SIAM J. Optim. \textbf{20} (2010), no. 6, 3280--3296.

\bibitem{Krug15} {\sc Kruger, A.Y.}, \textit{Error bounds and metric
subregularity}, Optimization \textbf{64} (2015), no. 1, 49--79.

\bibitem{Mord06} {\sc Mordukhovich, B.S.}, \textit{Variational analysis and
generalized differentiation. I. Basic theory}, Springer-Verlag, Berlin, 2006.

\bibitem{MoNaWa09} {\sc Mordukhovich, B.S.}, {\sc Nam, N.M.}, and {\sc
Wang, B.}, \textit{Metric regularity of mappings and generalized normals
to set images}, Set-Valued Var. Anal. \textbf{17} (2009), no. 4, 359--387.

\bibitem{Mord18} {\sc Mordukhovich, B.S.}, \textit{Variational analysis
and applications}, Springer, Cham, 2018.

\bibitem{Pang97} {\sc Pang, J.-S.}, \textit{Error bounds in mathematical
programming}, Math. Programming \textbf{79} (1997), no. 1-3, Ser. B, 299--332.

\bibitem{Peno04} {\sc Penot, J.-P.}, \textit{Calmness and stability properties
of marginal and performance functions}, Numer. Funct. Anal. Optim. \textbf{25}
(2004), no. 3-4, 287--308.

\bibitem{Peno13} {\sc Penot, J.-P.}, \textit{Calculus without derivatives},
Springer, New York, 2013.

\bibitem{Robi75} {\sc Robinson, S.M.}, \textit{An application of error
bounds for convex programming in a linear space}, SIAM J. Control
\textbf{13} (1975), 271--273.

\bibitem{Robi79} {\sc Robinson, S.M.}, \textit{Generalized equations and
their solutions. I. Basic theory. Point-to-set maps and mathematical
programming}, Math. Programming Stud. \textbf{10} (1979), 128--141.

\bibitem{Robi81} {\sc Robinson, S.M.}, \textit{Some continuity properties
of polyhedral multifunctions}, Math. Programming Stud. \textbf{14}
(1981), 206--214.

\bibitem{RocWet98} {\sc Rockafellar R.T.} and {\sc Wets, R. J.-B.},
\textit{Variational analysis}, Springer-Verlag, Berlin, 1998.

\bibitem{Uder12} {\sc Uderzo, A.}, \textit{On Lipschitz semicontinuity
properties of variational systems with application to parametric
optimization}, J. Optim. Theory Appl. \textbf{162} (2014), no. 1, 47--78.

\bibitem{Uder14b} {\sc Uderzo, A.}, \textit{On a quantitative semicontinuity
property of variational systems with applications to perturbed quasidifferentiable
optimization}. Constructive nonsmooth analysis and related topics, 115-–136,
Springer Optim. Appl., \textbf{87}, Springer, New York, 2014.

\bibitem{Uder17} {\sc Uderzo, A.}, \textit{On a set-covering property of
multivalued mappings}, Pure Appl. Funct. Anal. \textbf{2} (2017), no. 1,
129--151.

\bibitem{Uder18} {\sc Uderzo, A.}, \textit{An implicit multifunction theorem
for the hemiregularity of mappings with application to constrained optimization}.
Pure Appl. Funct. Anal. \textbf{3} (2018), no. 2, 371--391.

\bibitem{Uder19} {\sc Uderzo, A.}, \textit{On some generalized equations with metrically
$C$-increasing mappings: solvability and error bounds with applications to
optimization}, Optimization {\bf 68} (2019), 227--253.

\bibitem{Uder20} {\sc Uderzo, A.}, \textit{Solution analysis for a class
of set-inclusive generalized equations: a convex analysis approach},
Pure Appl. Funct. Anal. {\bf 5} (2020), no. 3, 769--790.

\bibitem{WuYe02} {\sc Wu, Z.} and {\sc Ye, J.J.}, \textit{On error bounds
for lower semicontinuous functions}, Math. Program. \textbf{92}
(2002), no. 2, Ser. A, 301--314.

%

\end{thebibliography}

\vskip1cm

\end{document}